\documentclass{amsart}

\usepackage{amsthm,amsxtra}
\usepackage{amssymb}
\usepackage{amsmath}
\usepackage{amsfonts}

\usepackage{graphicx}
\usepackage{xcolor}

\allowdisplaybreaks

\theoremstyle{plain}
\newtheorem{theorem}{Theorem}[section]
\newtheorem{proposition}[theorem]{Proposition}
\newtheorem{lemma}[theorem]{Lemma}

\theoremstyle{definition}
\newtheorem{definition}[theorem]{Definition}
\newtheorem{example}[theorem]{Example}

\theoremstyle{remark}
\newtheorem{remark}[theorem]{Remark}

\def\to{\rightarrow}

\def\Aut{\operatorname{Aut}}
\def\Hom{\operatorname{Hom}}

\def\Sym{\operatorname{Sym}}
\def\Ker{\operatorname{Ker}}
\def\dim{\operatorname{dim}}

\def\Pic{\operatorname{Pic}}
\def\Gr{\operatorname{Gr}}
\def\SL{\operatorname{SL}}

\def\SO{\operatorname{SO}}

\def\Sp{\operatorname{Sp}}

\begin{document}

\title[Standard embeddings of smooth Schubert varieties]
{Standard embeddings of smooth Schubert varieties in rational homogeneous manifolds \\ of Picard number 1}

\author{Shin-Young Kim and Kyeong-Dong Park}

\address{Shin-Young Kim \\ Korea Institute for Advanced Study \\ Seoul 02455, Korea}
\email{shinyoungkim@kias.re.kr}
\curraddr{Department of Mathematics and Statistics \\ Masaryk University \\ Kotl\'a\`rsk\'a 2 \\ 611 37 Brno, Czech Republic} 
\email{kim­@math­.muni.cz}

\address{Kyeong-Dong Park \\ Center for Geometry and Physics \\ Institute for Basic Science (IBS) \\ Pohang 37673, Korea}
% \\ Department of Mathematical Sciences \\ Seoul National University \\ Seoul 08826, Korea}
\email{kdpark@ibs.re.kr}
%kdpark@kias.re.kr}

\thanks{The first author was supported by the National Researcher Program 2010-0020413 of NRF, GA17-19437S of Czech Science Foundation(GACR), and was partially supported by the Simons-Foundation grant 346300 and the Polish Government MNiSW 2015-2019 matching fund. The second author was supported by BK21 PLUS SNU Mathematical Sciences Division and IBS-R003-Y1.}


\begin{abstract}
%Schubert varieties are classical varieties and one of the important examples among spherical varieties.
%Especially, all 
Smooth Schubert varieties in rational homogeneous manifolds of Picard number 1 are horospherical varieties.
We characterize standard embeddings of smooth Schubert varieties in rational homogeneous manifolds of Picard number 1
by means of varieties of minimal rational tangents.
In particular, we mainly consider nonhomogeneous smooth Schubert varieties in symplectic Grassmannians
and in the 20-dimensional $F_4$-homogeneous manifold associated to a short simple root.
\end{abstract}

\maketitle

\section{Introduction}

A \emph{rational homogeneous manifold} is a homogeneous space $G/P$
for a complex simple Lie group $G$ and a parabolic subgroup $P\subset G$.
Under the action of a Borel subgroup $B$ of $G$,
the closure of a $B$-orbit in $G/P$ is called a \emph{Schubert variety} of $G/P$.
For details about the parabolic subgroups and the Schubert varieties of $G/P$, see Springer \cite{Sp}.

Most Schubert varieties are singular,
and smooth Schubert varieties have been classified by using combinatorial and geometric methods
(for the combinatorial smoothness criterion, see Billey-Postnikov \cite{BP}).
Since conjugacy classes of parabolic subgroups of a simple Lie group are in one-to-one correspondence with subsets of
the set of simple roots (equivalently, nodes of the corresponding Dynkin diagram),
the Dynkin diagrams with a marked node %marking
correspond to rational homogeneous manifolds of Picard number 1.
A marked subdiagram of the marked Dynkin diagram of $G/P$ defines
a homogeneous submanifold $G_0/P_0$ of $G/P$, the $G_0$-orbit of the base point $eP \in G/P$,
then it is a smooth Schubert variety (see Section 2 of Hong-Mok \cite{HoM2}).
Lakshmibai-Weyman \cite{LW} and Brion-Polo \cite{BrP} showed that
when $G/P$ is a Hermitian symmetric space of compact type,
any smooth Schubert variety in $G/P$ is
a homogeneous submanifold associated to a subdiagram of the marked Dynkin diagram of $G/P$.
More generally, when $G/P$ is associated to a long simple root,
all smooth Schubert varieties are homogeneous submanifolds associated to subdiagrams of the marked Dynkin diagram (Proposition 3.7 of Hong-Mok \cite{HoM2}).
On the other hand, when $G/P$ is associated to a short simple root,
there may exist a smooth Schubert variety which is not homogeneous.
Recently, Hong \cite{Ho15} and Hong-Kwon \cite{HK} have classified all smooth Schubert varieties in this case.

A smooth Schubert variety $Z$ of $G/P$ is canonically embedded in $G/P$ by an equivariant embedding induced from the inclusion $B \subset G$.
By a \emph{standard embedding} of $Z$ into $G/P$, we will mean
the composite of the canonical equivariant embedding and an element of the automorphism group of $G/P$.
When $G/P$ is associated to a simple root and a homogeneous submanifold $G_0/P_0$ is not linear,
we have a characterization of standard embeddings of $G_0/P_0$ into $G/P$ by
means of varieties of minimal rational tangents as follows.

\begin{theorem}[Theorem 1.2 of Hong-Mok \cite{HoM1}, Theorem 1.2 of Hong-Park \cite{HP}]
\label{homogeneous submanifolds}
Let $X$ be a rational homogeneous manifold associated to a simple root
and let $Z$ be a nonlinear rational homogeneous manifold associated to a subdiagram of the marked Dynkin diagram of $X$.
If $f$ is a holomorphic embedding from a connected open subset $U$ of $Z$
into $X$ which respects varieties of minimal rational tangents for a general point $z\in U$,
then $f$ extends to a standard embedding of $Z$ into $X$.
\end{theorem}

Given a uniruled projective variety $X$ equipped with a minimal rational component $\mathcal K$,
the \emph{variety $\mathcal C_x(X) \subset \mathbb P(T_xX)$ of minimal rational tangents} (VMRT)
at a general point $x \in X$
is defined by the closure of the space of the tangent vectors of minimal rational curves
belonging to $\mathcal K$ passing through $x$.
If $X$ is a rational homogeneous manifold $G/P$ associated to a simple root, then there is a
canonical choice of a minimal rational component, namely, the
irreducible family of lines $\mathbb P^1$ which are contained in $X$
after we embed $X$ into $\mathbb P^{N}$ by the ample generator of
the Picard group of $X$. Similarly, there is a canonical choice of a
minimal rational component for a smooth Schubert variety.
For a general reference on the theory of rational curves and varieties of minimal rational tangents,
see Koll\'ar \cite{Ko}, Hwang-Mok \cite{HM99}, Hwang \cite{Hw00} and Mok \cite{Mk08}.

For a holomorphic embedding $f \colon U \rightarrow X$
from an open subset $U$ of a uniruled projective variety $Z$ with a minimal rational component $\mathcal H$,
we say that $f$ {\it respects varieties of minimal rational tangents} if
$$df(\mathcal C_z(Z))= df(\mathbb P(T_zZ)) \cap \mathcal C_{f(z)}(X)$$ for a general point $z \in U$,
where $\mathcal C_z(Z)$ is the variety of minimal rational tangents of $Z$ at $z \in Z$ associated to $\mathcal H$
and $\mathcal C_{f(z)}(X)$ is the variety of minimal rational tangents of $X$ at $f(z)$ associated to $\mathcal K$.

If $Z$ is linear, then the condition that $f \colon U \rightarrow X$ respects varieties of minimal rational tangent
is equivalent to the condition that $df(\mathbb P(T_zZ))$
should be contained in $\mathcal C_{f(z)}(X)$ for any $z \in U$.
In other words, for each $f(z) \in f(U)$
there is a linear space in $X$ which is tangent to $f(U)$ at $f(z)$.
In general, when $Z$ is linear, there is a non-standard embedding from an open subset $U$ of $Z$ into $X$
so that Theorem \ref{homogeneous submanifolds} does not hold.
For example, there is an embedding of $Z=\mathbb P^1$ into $X$
with $df(\mathbb P(T_zZ)) \subset \mathcal C_{f(z)}(X)$ for any $z \in \mathbb P^1$
whose image is not contained in any linear space in $X$
(see Section 6 of Choe-Hong \cite{CH}).
However, when $Z$ is a \emph{maximal} linear space,
Theorem 1.3 of Hong-Park \cite{HP} gives a related result with some exceptions involving counterexamples %had been
constructed by Choe-Hong \cite{CH}.

In this paper, we will prove a generalization of Theorem
\ref{homogeneous submanifolds}
in the case that $Z$ is a smooth Schubert variety of $G/P$
by the arguments developed in Hong-Mok \cite{HoM1} and Hong-Park \cite{HP}.
The fundamental tools for the proof are the non-equidimensional Cartan-Fubini type extension theorem
(Proposition \ref{Non-equidimensional Cartan-Fubini type extension theorem})
and the parallel transport of VMRTs along minimal rational curves (see Section 2.2).

\begin{theorem}
\label{Main theorem}
Let $X$ be a rational homogeneous manifold associated to a simple root and let $Z$ be a nonlinear smooth Schubert variety in $X$.
If $f$ is a holomorphic embedding from a connected open subset $U$ of $Z$
into $X$ which respects varieties of minimal rational tangents for a general point $z\in U$,
then $f$ extends to a standard embedding of $Z$ into $X$.
\end{theorem}

We denote the rational homogeneous manifold $G/P$ associated to a simple root $\alpha_i$ by $(G, \alpha_i)$.
Among rational homogeneous manifolds associated to a short simple root,
since $(B_\ell, \alpha_\ell)\cong (D_{\ell+1}, \alpha_\ell)$,
$(C_\ell, \alpha_1)\cong \mathbb P^{2\ell-1}$ and $(G_2,
\alpha_1)\cong (B_3, \alpha_1) \cong \mathbb Q^5$ as complex
manifolds, three cases can be regarded as rational homogeneous
manifolds associated to a long simple root.
Moreover, any smooth Schubert variety in the 15-dimensional $F_4$-homogeneous manifold $(F_4, \alpha_4)$ associated to the short simple root $\alpha_4$ is a linear space
by Theorem 1.3 of Hong-Kwon \cite{HK}.
Thus, it suffices to consider nonhomogeneous smooth Schubert varieties
in the symplectic Grassmannians (Section 3) and
in the 20-dimensional $F_4$-homogeneous manifold $(F_4, \alpha_3)$ associated to the short simple root $\alpha_3$ (Section 4).
Because these Schubert varieties are smooth nonhomogeneous horospherical varieties of Picard number 1,
we review notions and facts about horospherical varieties in Section 2.1.

\section{Horospherical varieties and Cartan-Fubini extension}

\subsection{Spherical and horospherical varieties}

For a complex reductive algebraic group $G$,
a complex algebraic variety with an action of $G$ is called a \emph{$G$-variety}.
A \emph{$G$-spherical variety} is a normal $G$-variety having an open orbit under the action of a Borel subgroup $B$ of $G$.
A %complex
normal %algebraic
$G$-variety is \emph{horospherical}
if $G$ acts with an open orbit $G/H$ isomorphic to
a torus bundle over a rational homogeneous manifold, or equivalently,
if the isotropy subgroup $H$ of a general point contains the unipotent radical of a Borel subgroup $B$.
The dimension of the torus fiber is called the \emph{rank} of a horospherical variety.

The Bruhat decomposition of $G$ implies that horospherical varieties are spherical (see Section 5.3 of Perrin \cite{Perrin}).
Toric varieties and rational homogeneous manifolds are the well-known examples of horospherical varieties.
%Since any Schubert variety is normal by Theorem 2.1.1 of Brion \cite{Br}, Schubert varieties of a rational homogeneous manifold $G/P$ are spherical.
Furthermore, we know that all smooth Schubert varieties in rational homogeneous manifolds of Picard number 1 are horospherical varieties
from the classification result of Hong-Mok \cite{HoM2}, Hong \cite{Ho15} and Hong-Kwon \cite{HK}.

Let $\{ \alpha_1, \cdots, \alpha_n\}$ be a %the
system of simple roots of $G$ following the standard numbering (e.g. Humphreys \cite{Humphreys})
and $P(\alpha_i)$ denote the maximal parabolic subgroup associated to a simple root $\alpha_i$.
%following the standard numbering (e.g. Humphreys \cite{Humphreys}).
For the corresponding system $\{ \omega_1, \cdots, \omega_n\}$ of fundamental weights,
$V(\omega_i)$ denotes the irreducible $G$-representation space
with the $i$-th fundamental weight $\omega_i$ as a highest weight.
When we take a highest weight vector $v_i$ of $V(\omega_{i})$,
the $G$-orbit of $[v_i]$ in $\mathbb P( V(\omega_i))$ is closed and isomorphic to the rational homogeneous manifold $G/P(\alpha_i)$
which is denoted by $(G, \alpha_i)$.

If $v_i$ and $v_j$ are highest weight vectors of
$V(\omega_i)$ and $V(\omega_j)$ respectively,
we will consider the closure of the $G$-orbit of the point $[v_i + v_j]$
in $\mathbb P( V(\omega_i) \oplus V(\omega_j) )$.
For any $i \neq j$,
the open orbit $G.[v_i + v_j]$ is isomorphic to a $\mathbb C^*$-bundle over a rational homogeneous manifold $G/(P(\alpha_i)\cap P(\alpha_j))$.
According to Propositon 2.1 of Hong \cite{Ho16},
since the closure of $G.[v_i + v_j]$ in $\mathbb P( V(\omega_i) \oplus V(\omega_j) )$ is a normal variety,
$\overline{G.[v_i + v_j]}$ is a horospherical $G$-variety and we denote it by $(G, \alpha_i, \alpha_j)$.
The smooth projective horospherical varieties of Picard number 1 are classified by Pasquier \cite{Pasquier}
using the fact that any nonlinear smooth horospherical variety of Picard number 1 is of the form $(G, \alpha_i, \alpha_j)$.

\begin{proposition}[Theorem 0.1 of Pasquier \cite{Pasquier}]
\label{horosphercal}
Let $G$ be a connected reductive algebraic group.
A smooth projective horospherical $G$-variety $X$ of Picard number 1
is either homogeneous or horospherical of rank 1.
In the nonhomogeneous case,
its automorphism group $\Aut(X)$ is a connected non-reductive linear algebraic group acting with exactly two orbits $X_0$ and $Z$;
moreover, $X$ is uniquely determined by
its two closed $G$-orbits $Y\subset X_0$ and $Z$,
isomorphic to rational homogeneous manifolds $G/P_Y$ and $G/P_Z$, respectively,
where $(G, P_Y, P_Z)$ is one of the triples in the following list:
\begin{enumerate}%\setlength{\baselineskip}{0.85 \baselineskip}
\item[\rm (1)] $(B_n, P(\alpha_{n-1}), P(\alpha_n))$ with $n \geq 3$;
\item[\rm (2)] $(B_3, P(\alpha_1), P(\alpha_3))$;
\item[\rm (3)] $(C_n, P(\alpha_{k}), P(\alpha_{k-1}))$
 with $n \geq k \geq 2$;
\item[\rm (4)] $(F_4, P(\alpha_2), P(\alpha_3))$;
\item[\rm (5)] $(G_2, P(\alpha_2), P(\alpha_1))$.
\end{enumerate}
\end{proposition}

\begin{proposition}[Theorem 1.11 of Pasquier \cite{Pasquier}]
\label{automorphism groups of horosphercal}
In the above cases (1) -- (5), the automorphism group of $X$ is isomorphic to
$(\SO(2n+1) \times \mathbb C^*) \ltimes V(\omega_n)$,
$(\SO(7) \times \mathbb C^*) \ltimes V(\omega_3)$,
$((\Sp(2n) \times \mathbb C^*)/\{\pm 1\}) \ltimes V(\omega_1)$,
$(F_4 \times \mathbb C^*) \ltimes V(\omega_4)$ and
$(G_2 \times \mathbb C^*) \ltimes V(\omega_1)$, respectively.
\end{proposition}

Recently, Hong \cite{Ho16} showed that a smooth horospherical variety $X$ of Picard number 1
can be embedded as a linear section into a rational homogenous manifold of Picard number 1
except when $X$ is $(B_n, \alpha_{n-1}, \alpha_n)$ for $n \geq 7$.
For a description of their tangent space based on weights and roots, see Proposition 2.6 of Kim \cite{Kim}.

\begin{example} [Odd symplectic Grassmannian $(C_n, \alpha_{k}, \alpha_{k-1})$]
Let $V$ be a complex vector space endowed with
a skew-symmetric bilinear form $\omega$ of maximal rank.
We denote the variety of all $k$-dimensional isotropic subspaces in $V$ by
$\Gr_{\omega}(k, V)=\{ W \subset V %\mid
: \dim W=k, \, \omega|_{W} \equiv 0 \}.$
When $\dim V$ is even, say, $2n$, the form $\omega$ is a nondegenerate symplectic form
and this variety $\Gr_{\omega}(k, 2n)$ is the usual symplectic Grassmannian,
which is homogeneous under the action of the symplectic group $\Sp(2n)$.
But when $\dim V$ is odd, say, $2n+1$,
the skew-form $\omega$ has a one-dimensional kernel.
The variety $\Gr_{\omega}(k, 2n+1)$,
called the \emph{odd symplectic Grassmannian},
is not homogeneous and
has two orbits under the action of its automorphism group if $2 \leq k \leq n$
(cf. Mihai \cite{Mihai} and Proposition 1.12 of Pasquier \cite{Pasquier}).
If $k=1$, then the isotropic condition holds trivially
so that $\Gr_{\omega}(1, V)$ is just the linear space $\mathbb P^{\dim V-1}$.
Next, for $k=n+1$ the odd symplectic Grassmannian $\Gr_{\omega}(n+1, 2n+1)$ is isomorphic to the symplectic Grassmannian $\Gr_{\omega}(n, 2n)$
because any $(n+1)$-dimensional isotropic subspace
must contain the one-dimensional kernel of $\omega$.
In what follows, we will assume that $2\leq k \leq n$
when considering the odd symplectic Grassmannians.

Let $S$ be an odd symplectic Grassmannian $\Gr_{\omega}(k, 2n+1)$ for $2\leq k \leq n$.
Then $S$ is a smooth Fano manifold of Picard number 1
and the automorphism group $\Aut(S)$ of $S$ is isomorphic to the semi-direct product
$((\Sp(2n) \times \mathbb C^*)/\{\pm 1\}) \ltimes \mathbb C^{2n}$.
We know that $S$ has two orbits under its automorphism group. The closed orbit
$\{ W \in \Gr_{\omega}(k, 2n+1) : % \mid
\Ker \omega \subset W\}$
is isomorphic to the symplectic Grassmannian $\Gr_{\omega}(k-1, 2n)$
and the open orbit $\{ W \in \Gr_{\omega}(k, 2n+1) : % \mid
\Ker \omega \not\subset W\}$
is isomorphic to the dual tautological sub-bundle
on the symplectic Grassmannian $\Gr_{\omega}(k, 2n)$.
In fact, choosing a supplementary subspace $V' \subset V$ so that $V=\Ker \omega \oplus V'$,
%In fact, considering the decomposition $\mathbb C^{2n+1}=\Ker \omega \oplus \mathbb C^{2n}$,
any $W \in \Gr_{\omega}(k, V) = \Gr_{\omega}(k, 2n+1)$ containing $\Ker \omega$
corresponds a point of $\Gr_{\omega}(k-1, V') = \Gr_{\omega}(k-1, 2n)$.
And the projection coming from the above decomposition gives
a map from the open orbit onto $\Gr_{\omega}(k, 2n)$ of which the fiber at a point $E \in\Gr_{\omega}(k, 2n)$ is $E^*$
(for details, see Proposition 4.3 of Mihai \cite{Mihai}).
Consequently, the odd symplectic Grassmannian $\Gr_{\omega}(k, 2n+1)$
has three orbits under the semisimple part $\Sp(2n)$ of its automorphism group.
In particular, the $\Sp(2n)$-closed orbit lying in the open orbit
is isomorphic to a symplectic Grassmannian $\Gr_{\omega}(k, 2n)$.
\end{example}

\subsection{Second fundamental form and Cartan-Fubini extension}

Let $V$ be a finite-dimensional vector space and let $\mathcal A
\subset \mathbb P(V)$ be a complex-analytic subvariety. Denote by
$\widetilde{\mathcal A} \subset V \backslash \{0\}$ the \emph{affine cone}
of $\mathcal A$, i.e., the pre-image $\pi^{-1}(\mathcal A)$ of the
canonical projection $\pi\colon V \backslash\{0\} \rightarrow
\mathbb P(V)$.
For a smooth point $\eta \in \widetilde{\mathcal A}$, the \emph{second fundamental form}
$$\sigma_{\eta} \colon T_{\eta} \widetilde{\mathcal A}
\times T_{\eta} \widetilde{\mathcal A}\rightarrow V/ T_{\eta}
\widetilde{\mathcal A}$$ of $\widetilde{\mathcal A}\subset V$ at
$\eta \in \widetilde{\mathcal A}$ is defined by $\sigma_{\eta}(\xi,
\zeta)=\nabla_{\xi}\hat{\zeta} \mod T_{\eta} \widetilde{\mathcal A}$
for any $\xi, \zeta \in T_{\eta} \widetilde{\mathcal A}$, where
$\hat{\zeta}$ is a local vector field with
$\hat{\zeta}(\eta)=\zeta$, and $\nabla$ is the Euclidean flat
connection on the Euclidean space $V$. Another definition is given
by the differential of the \emph{Gauss map}
$$\Gamma \colon \widetilde{\mathcal A} \rightarrow \Gr(d, V), \quad
\beta \in \widetilde{\mathcal A} \mapsto
[T_{\beta}\widetilde{\mathcal A}] \in \Gr(d, V)$$ at $\eta$,
where $d=\dim \mathcal A +1$.
The differential of the Gauss map $\Gamma$
at $\eta \in \widetilde{\mathcal A}$ is a linear map
$$d\Gamma_{\eta} : T_{\eta}\widetilde{\mathcal A} \rightarrow
T_{[T_{\eta}\widetilde{\mathcal A}]}\Gr(d, V) \cong
\Hom(T_{\eta}\widetilde{\mathcal A}, V / T_{\eta}\widetilde{\mathcal A}).$$
The canonical isomorphism $\chi$ between the
tangent space $T_{[W]} \Gr(d, V)$ of a Grassmannian and $\Hom(W,
V/W)$ is given by $\xi\mapsto \chi_{\xi}$ with $\chi_{\xi}(w):=
\rho'(0)+W$, where $\rho\colon D\to V$ is a \emph{moving vector
field} with $\rho(0)=w$ along a holomorphic path $\hat{\xi}$ from a
connected open subset $D\subset \mathbb C$ into $\Gr(d, V)$ such
that $\hat{\xi}(0)=[W]$ and $\hat{\xi}'(0)=\xi$.
Here, $\rho\colon
D\to V$ is called a moving vector field along a holomorphic path
$\hat{\xi}$ if $\rho(t)\in \hat{\xi}(t)$ for every $t\in D$.
If we use this canonical isomorphism, the differential of the Gauss map is
described as follows.
For $\xi, \zeta \in T_{\eta} \widetilde{\mathcal A}$ we choose the following gadgets:

\begin{itemize}%\setlength{\baselineskip}{0.85 \baselineskip}
  \item a holomorphic path $\alpha \colon D \to \widetilde{\mathcal A}$
  with $\alpha(0) = \eta$ and $\alpha'(0) = \xi$,
  \item a vector field $\rho\colon D\to V$ along $\alpha$ with
  $\rho(0)=\zeta$,
  i.e., $\rho(t)\in T_{\alpha(t)}\widetilde{\mathcal A}$ for every $t\in D$.
\end{itemize}
If we set $\hat{\xi}:=\Gamma \circ \alpha \colon D \to \Gr(d, V)$,
then $\hat{\xi}(0)=[T_{\eta} \widetilde{\mathcal A}]$ and
$\hat{\xi}'(0)=d\Gamma_{\eta}(\xi)\in
\Hom(T_{\eta}\widetilde{\mathcal A}, V / T_{\eta}\widetilde{\mathcal
A})$ under the canonical isomorphism. Since $\rho$ is a moving
vector field along $\hat{\xi}$,
$$d\Gamma_{\eta}(\xi)(\zeta)=\hat{\xi}'(0)(\zeta)
=\rho'(0)+T_{\eta}\widetilde{\mathcal A} \in V
/T_{\eta}\widetilde{\mathcal A}.$$ Restating the above construction,
we have obtained a symmetric bilinear map, the second fundamental
form, $\sigma_{\eta}(\alpha'(0), \rho(0))=\rho'(0)+T_{\eta}\widetilde{\mathcal A}$
for every holomorphic path $\alpha$ in $\widetilde{\mathcal A}$ with
$\alpha(0)=\eta$ and every moving vector field $\rho$ along
$\alpha$.

For a subspace $E$ of $T_{\eta} \widetilde{\mathcal A}$
we define
$$\Ker \sigma_{\eta}(\,\cdot\,, E ):=
\{ \zeta \in T_{\eta}\widetilde{\mathcal A} : \sigma_{\eta}(\zeta ,
\xi)=0, \, \forall \xi \in E \}.$$
From the fact that $\widetilde{\mathcal A}$ is a cone with the vertex at $0$,
it follows that $\sigma_{\eta}(\eta, \xi)=0$ for any $\xi \in
T_{\eta}\widetilde{\mathcal A}$. In particular, $\mathbb C \eta$ is
contained in $\Ker \sigma_{\eta}(\,\cdot\,, E)$ for any subspace $E$
of $T_{\eta} \widetilde{\mathcal A} $. At $[\eta]=\pi(\eta) \in
\mathcal A$ the tangent space $T_{[\eta]}\mathcal A$ is given by
$(T_{\eta}\widetilde{\mathcal A}/\mathbb C \eta) \otimes (\mathbb C
\eta)^*$. Thus the second fundamental form $\sigma_{\eta} \colon
T_{\eta} \widetilde{\mathcal A} \times T_{\eta} \widetilde{\mathcal
A} \rightarrow V/ T_{\eta} \widetilde{\mathcal A} $ of
$\widetilde{\mathcal A}$ at $\eta$ induces the projective second fundamental
form $\overline{\sigma}_{[\eta]}\colon T_{[\eta]}\mathcal A \times
T_{[\eta]}\mathcal A \rightarrow T_{[\eta]}\mathbb
P(V)/T_{[\eta]}\mathcal A$ of $\mathcal A$ at $[\eta]$.
From now on we will use the notation $\sigma_{[\eta]}$ instead of $\overline{\sigma}_{[\eta]}$
for the sake of convenience.
For a subspace $\overline{E}$ of $T_{[\eta]}\mathcal A$ we define $\Ker
\sigma_{[\eta]}(\,\cdot\,, \overline{E} )$ by $\{ \overline{\zeta}
\in T_{[\eta]} \mathcal A : \sigma_{[\eta]}(\overline{\zeta},
\overline{\xi})=0,\, \forall \overline{\xi} \in \overline{E} \}$.

\begin{definition}%(Section 2.1 of Hong-Mok \cite{HoM1})
Let $(X, \mathcal K)$ and $(Z, \mathcal H)$ be two polarized uniruled projective manifolds
equipped with a minimal rational component.
For a holomorphic %embedding
immersion $f \colon U \rightarrow X$ from an open subset $U$ of $Z$ in the analytic topology,
%Assume that a holomorphic embedding $f \colon U \rightarrow X$
%respects varieties of minimal rational tangents from an open subset $U$ of $Z$ in the analytic topology.
we say that $f$ is {\it nondegenerate with respect to $(\mathcal K, \mathcal H)$} if
\begin{enumerate}
\item[{\rm (1)}] its image $f(U)$ is not contained in the bad locus of $\mathcal K$,
which means the smallest subvariety $B$ of $X$ such that
for all $x \in X \backslash B$, any minimal rational curve passing through $x$ is free and a general minimal rational curve passing through $x$ is standard, and
%i.e., for all $x \in f(U)$, any minimal rational curve passing through $x$ is free and a general minimal rational curve passing through $x$ is standard, and

\item[{\rm (2)}] for a general point $z \in U$ and a general smooth point $\alpha \in \widetilde{\mathcal C}_z(Z)$,
$df(\alpha)$ is a smooth point of $\widetilde{\mathcal C}_{f(z)}(X)$
and $$\Ker \sigma_{df(\alpha)}( \,\cdot \, ,T_{df(\alpha)}(df(\widetilde{\mathcal C}_z(Z) )))=\mathbb C df(\alpha),$$
where $\sigma_{df(\alpha)}%(\,,\,)
$ denotes the second fundamental form of the affine cone
$\widetilde{\mathcal C}_{f(z)}(X) \subset T_{f(z)}X
$ at $df(\alpha)$.
\end{enumerate}
\end{definition}

Now, as the main ingredient for the proof of Theorem \ref{Main theorem},
we state the non-equidimensional Cartan-Fubini type extension theorem,
which says the rational extension of germs of holomorphic maps
respecting varieties of minimal rational tangents.
For an introductory exposition on an analytic continuation along minimal rational curves and Cartan-Fubini extension,
we refer to Section 2 of Mok \cite{Mk16}.

\begin{proposition}[Theorem 1.1 of Hong-Mok \cite{HoM1}]
\label{Non-equidimensional Cartan-Fubini type extension theorem}
Let
$(X,\mathcal K)$ and $(Z,\mathcal H)$ be two uniruled projective
manifolds equipped with a minimal rational component. Assume that
$Z$ is of Picard number 1 and that $\mathcal C_z(Z)$ is
positive-dimensional at a general point $z \in Z$. Let $f \colon U
\to X$ be a holomorphic immersion defined on a connected open subset
$U \subset Z$. If $f$ respects varieties of minimal rational
tangents and is nondegenerate with respect to $(\mathcal K,\mathcal
H)$, then $f$ extends to a rational map $F \colon Z \to X$.
\end{proposition}

To use this result, we need to compute the second fundamental form of
the variety of minimal rational tangents as subvariety in the projective tangent space
and to check the nondegeneracy of the pair of varieties of minimal rational tangents
(Proposition \ref{nondegenerate: symplectic} and Proposition \ref{nondegenerate: F_4}).
Then we can apply the non-equidimensional Cartan-Fubini type extension theorem
and get a rational extension $F\colon Z\rightarrow X$ of $f$.
Up to the action of $\Aut(X)$,
$F(x_0)=x_0$ and $ \mathcal C_{x_0}(F(Z))=\mathcal C_{x_0}(Z)$
for a fixed general point $x_0 \in U \subset Z$.
Since $f$ sends minimal rational curves in $Z$ to minimal rational curves in $X$
and the tangency property of the two VMRTs of $Z$ and $F(Z)$ at an intersection point
does imply equality of these VMRTs in the case of smooth Schubert varieties in a rational homogeneous manifold of Picard number 1,
as established in the next sections,
we can extend the map inductively along minimal rational curves.
Consequently, $F$ is the identity map up to the action of $\Aut(X)$.

\section{Smooth Schubert varieties in symplectic Grassmannians}

Let $G$ be a connected simple Lie group of type $C_{\ell}$ and
let $X$ be a rational homogeneous manifold $G/P$ associated to a simple root $\alpha_k$ ($1 \leq k \leq \ell$).
Then $X$ is the \emph{symplectic Grassmannian} $\Gr_{\omega}(k, 2\ell)$ of isotropic $k$-subspaces in $V=\mathbb C^{2\ell}$
with respect to a \emph{symplectic form} $\omega$ on $\mathbb C^{2\ell}$,
where the symplectic form means a nondegenerate skew-symmetric bilinear form on $V$.
Take a basis $\{ e_1, \, \cdots, e_{2 \ell} \}$ of $V$
such that $\omega(e_{\ell-i}, e_{\ell+i+1}) = - \omega(e_{\ell+i+1}, e_{\ell-i}) = 1$ for $0 \leq i \leq \ell-1$,
and all other $\omega(e_i, e_j)$ are zero.
Define $F_j \subset V$ as the subspace generated by $e_1, \, \cdots, e_j$ for $1 \leq j \leq 2\ell$ and set $F_0=\{0\}$.
Then $F^{\perp}_{\ell-i} = F_{\ell+i}$ for $0 \leq i \leq \ell$.
The symplectic group $G=\Sp(V)$ naturally acts on $\Gr_{\omega}(k, 2\ell)$
and the parabolic subgroup $P$ is the isotropy subgroup of $G$ at $[F_k]$.
If $k=1$, then the isotropic condition holds trivially
so that $\Gr_{\omega}(1, 2\ell)$ is just the linear space $\mathbb P^{2\ell-1}$.
On the other hand, if $k=\ell$, a rational homogeneous manifold associated to the long simple root $\alpha_{\ell}$ is
the Lagrangian Grassmannian $\Gr_{\omega}(\ell, 2\ell)$
of which any smooth Schubert variety is a homogeneous submanifold associated to a subdiagram of the marked Dynkin diagram of $\Gr_{\omega}(\ell, 2\ell)$
by Lakshmibai-Weyman \cite{LW}, Brion-Polo \cite{BrP} and Hong-Mok \cite{HoM2}.
In what follows, we will assume that $1 < k < \ell$.

Fix a $k$-dimensional isotropic subspace $E \subset V$.
Since we can view $X$ as a subvariety of the
Grassmannian $\Gr(k, V)$ of $k$-dimensional subspaces in $V$ and
the tangent space of %a Grassmannian
$\Gr(k, V)$ at $[E]$ is naturally isomorphic to
$\Hom(E, V/E)$, %=E^* \otimes V/E$,
we have
$$T_{[E]} X = \{ h \in \Hom(E, V/E) \colon
\omega(h(e_1), e_2) + \omega(e_1, h(e_2)) = 0, \, \forall e_1, e_2
\in E \}.$$
Putting $E^{\perp} := \{v \in V : \omega(v, e) = 0,\,
\forall e \in E \}$, $E^{\perp}$ is a subspace of dimension $2\ell-k$
containing $E$ because $E$ is an isotropic subspace.
From the
nondegeneracy of $\omega$, the isomorphism $V/E^{\bot} \cong E^*$ is
induced by the symplectic form $\omega$.
Then, under the map $\psi \colon E^* \otimes V/E \to E^* \otimes V/E^{\perp} \cong E^* \otimes E^*$ which is composition of projection and the isomorphism $V/E^{\bot} \cong E^*$,
the tangent space of the symplectic Grassmannian $\Gr_{\omega}(k, 2\ell)$ at $[E]$ is
the inverse image $\psi^{-1}(S^2 E^*)$ of the symmetric square $S^2 E^* \subset E^* \otimes E^*$
and can be identified with
$$T_{[E]} \Gr_{\omega}(k, 2\ell)=(E^* \otimes (E^{\bot}/E))\oplus S^2 E^*. $$

Minimal rational curves of $X$ are lines of $\Gr(k, 2\ell)$ lying on $X$.
Thus, the variety $\mathcal C_{[E]}(X)$ of minimal rational tangents of $X$ at a
point $[E]\in X$ is the variety of decomposable tensors in $T_{[E]} X$.
%From now on, we take the natural bilinear pairing $\langle \, , \, \rangle \colon E^* \times E \to \mathbb C$ given by $\langle f^* , e \rangle = f^*(e)$, which gives the duality correspondence $e^* \mapsto e$ between $E^*$ and $E$.
From now on, we take the standard inner product on $E^*$ associated with Lie group $\SO(E^*)$, which gives the correspondence $e^* \mapsto e$ between $E^*$ and $E$.
% Let  $<,>\colon E \times E \rightarrow \mathbb C$ be the standard inner product on $E$ associated with Lie group $\SO(E)$, let $E^*$ be  the dual of $E$ and let $f$ be the dual of $e^*$.
If a decomposable tensor $h=e^*\otimes v$ is contained in $T_{[E]} X \subset E^* \otimes V/E$, then
\begin{eqnarray*}
\omega(v, e')
&=& \omega(h(e), e') = -\omega(e, h(e')) \\
&=& -\omega(e, (e^*\otimes v)e') = -\omega(e, e^*(e')v) \\
&=& -\omega(e, v) e^*(e') \quad \mbox{for all } e'\in E,
\end{eqnarray*}
that is, $\omega(v, \cdot) |_E %= v^{\flat}
\in \mathbb C e^*$.
Conversely, if $\omega(v, \cdot)|_E \in \mathbb C e^*$,
then $e^*\otimes v$ is contained in $T_{[E]} X$.
Therefore, the affine cone of
%the variety
$\mathcal C_{[E]}(X)$
%of minimal rational tangents of $X$
is
$$\widetilde{\mathcal C}_{[E]}(X)=\{e^*\otimes v\in E^*\otimes(V/E)
: %v^{\flat}
\omega(v, \cdot) |_E \in \mathbb C e^* \} \backslash \{0\}.$$

By Proposition 3.2.1 of Hwang-Mok \cite{HM05} or Corollary 5.5 of Landsberg-Manivel \cite{LM},
the variety $\mathcal A$ of minimal rational tangents of $\Gr_{\omega}(k, 2\ell)$ at a
point $[E]\in \Gr_{\omega}(k, 2\ell)$ is the projectivization of the affine cone
\begin{eqnarray*}
\widetilde{\mathcal A} = \{ u \otimes q + c u^2 : u \in U, q \in Q, c \in \mathbb C \} \backslash \{0\}
 \subset (U \otimes Q) \oplus S^2U,
\end{eqnarray*}
where $U=E^*$ and $Q=E^{\perp}/E$.
Under the projection $\mathcal A \rightarrow \mathbb P(U)=\mathbb P^{k-1}$
 defined by $u \otimes q + cu^2 \mapsto u$,
 $\mathcal A$ becomes a $\mathbb P^{2m}$-bundle over $\mathbb P^{k-1}$, where $m=\ell-k$.

For integers $a, b$ with $0 \leq a < k < b \leq 2\ell-a$,
define $$\Gr_{\omega}(k, 2\ell; F_a, F_b):=\{ E \in \Gr_{\omega}(k, 2\ell) : F_a \subset E \subset F_b \},$$
where $F_j \subset V$ is the subspace generated by $e_1, \cdots, e_j$.
Recently, Hong \cite{Ho15} has classified smooth Schubert varieties in the symplectic Grassmannian $\Gr_{\omega}(k, 2\ell)$.
From this result, all smooth Schubert varieties are of this form satisfying certain condition:

\begin{lemma} \label{smooth Schubert variety: symplectic}
Smooth Schubert varieties of the symplectic Grassmannian $\Gr_{\omega}(k, 2\ell)$
are of the form $\Gr_{\omega}(k, 2\ell; F_a, F_b)$,
where one of the following holds:
\begin{enumerate}
\item[{\rm (1)}] $0 \leq  a < k$ and $(k < b \leq  \ell \mbox{ or } b = 2\ell - a)$;
a homogeneous submanifold associated to a subdiagram of the marked Dynkin diagram
corresponding to the symplectic Grassmannian $\Gr_{\omega}(k, 2\ell)$,

 \begin{picture} (180,40)
 \put(30,20){\circle{5}} %1
 \put(50,20){\circle{5}} %2
 \put(65,18){$\cdot$} %3
 \put(70,18){$\cdot$} %3
 \put(75,18){$\cdot$} %3
 \put(90,20){\circle{5}} %4
 \put(106,17){$\times$} %5
 \put(130,20){\circle{5}} %6
 \put(145,18){$\cdot$} %7
 \put(150,18){$\cdot$} %7
 \put(155,18){$\cdot$} %7
 \put(170,20){\circle{5}} %8
 \put(190,20){\circle{5}} %9

{\color{blue}
 \put(77,5){\line(0,1){30}} %box
 \put(77,5){\line(1,0){83}}
 \put(160,35){\line( 0,-1){30}}
 \put(160,35){\line(-1,0){83}}
 }

 \put(32.5,20){\line(1,0){15}} %1 edges
 \put(52.5,20){\line(1,0){10}} %2
 \put(79.5,20){\line(1,0){8}} %3
 \put(92.5,20){\line(1,0){15}} %4
 \put(112.5,20){\line(1,0){15}} %5
 \put(132.5,20){\line(1,0){10}} %6
 \put(159.5,20){\line(1,0){8}} %7
 \put(172.5,19){\line(1,0){15}} %4
 \put(172.5,21){\line(1,0){15}} %4

 \put(179,17.5){$<$} %4 arrow

 \put(220,20){$(\mbox{when } k<b\leq \ell)$}
 \end{picture}

 \begin{picture} (180,40)
 \put(30,20){\circle{5}} %1
 \put(50,20){\circle{5}} %2
 \put(65,18){$\cdot$} %3
 \put(70,18){$\cdot$} %3
 \put(75,18){$\cdot$} %3
 \put(90,20){\circle{5}} %4
 \put(106,17){$\times$} %5
 \put(130,20){\circle{5}} %6
 \put(145,18){$\cdot$} %7
 \put(150,18){$\cdot$} %7
 \put(155,18){$\cdot$} %7
 \put(170,20){\circle{5}} %8
 \put(190,20){\circle{5}} %9

{\color{blue}
 \put(77,5){\line(0,1){30}} %box
 \put(77,5){\line(1,0){123}}
 \put(200,35){\line( 0,-1){30}}
 \put(200,35){\line(-1,0){123}}
 }

 \put(32.5,20){\line(1,0){15}} %1 edges
 \put(52.5,20){\line(1,0){10}} %2
 \put(79.5,20){\line(1,0){8}} %3
 \put(92.5,20){\line(1,0){15}} %4
 \put(112.5,20){\line(1,0){15}} %5
 \put(132.5,20){\line(1,0){10}} %6
 \put(159.5,20){\line(1,0){8}} %7
 \put(172.5,19){\line(1,0){15}} %4
 \put(172.5,21){\line(1,0){15}} %4

 \put(179,17.5){$<$} %4 arrow

 \put(220,20){$(\mbox{when } b=2\ell-a)$}
 \end{picture}
\item[{\rm (2)}] $0 \leq a < k$ and $b = 2\ell - a - 1$; an odd symplectic Grassmannian $(C_{\ell-1}, \alpha_{k-a}, \alpha_{k-a-1})$,
\item[{\rm (3)}] $a = k - 1$ and $\ell + 1 \leq b \leq 2\ell - k %1
$; a linear space $\mathbb P^{b-k}$.
\end{enumerate}
\end{lemma}

\begin{proof}
Proposition 3.1 and Proposition 4.7 of Hong \cite{Ho15}.
\end{proof}

As we have seen in Example 2.2,
the odd symplectic Grassmannian $\Gr_{\omega}(k, 2\ell; F_a, F_{2\ell-a-1})$ is not homogeneous
but a smooth Schubert variety of the symplectic Grassmannian $\Gr_{\omega}(k, 2\ell)$.
To prove Theorem \ref{Main theorem} in the case that $X$ is the symplectic Grassmannian $\Gr_{\omega}(k, 2\ell)$,
it suffices to consider when $Z$ is an odd symplectic Grassmannian.
In the remaining of the section, we will prove Theorem \ref{Main theorem} in this case.

\begin{lemma} \label{second fundamental form}
Let $X$ be the symplectic Grassmannian $\Gr_{\omega}(k, 2\ell)$
with $1 < k < \ell$ and
$\mathcal A$ be the variety of minimal rational tangents of $X$ at a
point $[E]\in X$. The tangent space $T_{\beta}$ of
$\widetilde{\mathcal A}$ at $\beta \in \widetilde{\mathcal A}$ is
given by
\begin{eqnarray*}
%T_{\alpha} &=& \{ u \otimes q' + u \circ u' : u' \in U, q' \in Q \} \text{ for } \alpha = u^2 \\
T_{\beta} &=&\{ u \otimes q' + u' \otimes q + 2 u \circ u' : u' \in U, q' \in Q \}
\text{ if } \beta = u \otimes q + u^2, \\
T_{\beta} &=& \{ u \otimes q' + u' \otimes q + c
%u \circ u
u^2 : u' \in U, q' \in Q, c \in \mathbb C \} \text{ if } \beta = u \otimes q.
\end{eqnarray*}
The second fundamental form $\sigma \colon T_{\beta} \times T_{\beta}
\longrightarrow (T_{[E]}X)/T_{\beta} $ of $\widetilde{\mathcal A} \subset T_{[E]} X$
at $\beta \in \widetilde{\mathcal A}$ is given as follows:
\begin{enumerate}
%\item for $\alpha=u^2$,
%\begin{eqnarray*}
%&&\sigma(u \circ u', u \circ u'')=u' \circ u'' \\
%&&\sigma(u \circ u', u \otimes q')=u' \otimes q' \\
%&&\sigma(u \otimes q', u \otimes q'')=0
%\end{eqnarray*}
\item[{(I)}] for $\beta=u \otimes q + u^2$,
\begin{eqnarray*}
%\sigma\,\,:\,\, T_{\beta} \times T_{\beta}\qquad\qquad &\longrightarrow& V/T_{\beta} \\
&&\sigma(u' \otimes q + 2 u \circ u',\,\, u \otimes q')= u' \otimes q' \\
&&\sigma(u' \otimes q + 2 u \circ u',\,\, u'' \otimes q + 2 u \circ u'') = 2 u' \circ u'' \\
&&\sigma(u \otimes q', \,\,u \otimes q'')=0 {\rm;}
\end{eqnarray*}
\item[{(II)}] for $\beta = u \otimes q$,
\begin{eqnarray*}
%\sigma\,\,:\,\, T_{\beta} \times T_{\gamma}&\longrightarrow& V/T_{\gamma} \\
&&\sigma(u' \otimes q , u \otimes q')= u' \otimes q' \\
&&\sigma(u' \otimes q, u ^2) =  2u \circ u' \\
&&\sigma(u' \otimes q , u''\otimes q ) = 0 \\
&&\sigma(u \otimes q', u \otimes q'') =0 \\
&&\sigma(u \otimes q', u^2) = 0   \\
&&\sigma(u^2, u^2)  = 0,
\end{eqnarray*}
\end{enumerate}
where $u', u'' \in U$ and $q', q'' \in Q$.
\end{lemma}

\begin{remark} The second fundamental form $\sigma$ of
$\widetilde{\mathcal A}$ at $\beta \in \widetilde{\mathcal A}$ has
its image in the quotient space $(T_{[E]}X)/T_{\beta}$.  For simplicity, here
and henceforth we will use the same notation for an element $v \in T_{[E]} X$
and its image in the quotient $(T_{[E]}X)/T_{\beta}$.
We will use the same convention for the second fundamental forms of other subvarieties.
\end{remark}

\begin{proof}
This result is given in Lemma 3.2 of Hong-Park \cite{HP} without details.
We give the details of the proof.
First, to obtain the tangent space $T_{\beta} \widetilde{\mathcal A}$,
we consider the velocity vectors of curves in the affine cone
$\widetilde{\mathcal A}$.
Let $\{u_t\} \subset U$ be a curve with $u_0=u$ and $\{q_t\} \subset
Q$ be a curve with $q_0=q$.
The curves $u_t \otimes q +u_t^2$, $u \otimes q_t + u^2$ in the affine cone $\widetilde{\mathcal A}$
pass through a point $u \otimes q + u^2$
and their velocity vectors are
$u'\otimes q + 2 u \circ u'$ for some $u' \in U$
and $u \otimes q'$ for some $q' \in Q$, respectively.
Because $\dim \widetilde{\mathcal A}=k+2m=\dim U+\dim Q$,
the tangent space $T_{\beta} \widetilde{\mathcal A}$ at a point $\beta = u \otimes q + u^2$
is spanned by the vectors $\{u' \otimes q + 2 u \circ u' : u' \in U \}$ and $\{ u\otimes q' : q' \in Q \}$.
Similarly, the curves $u_t \otimes q$ and $u \otimes q_t$ pass through a point $u \otimes q$ when $t=0$
so that their velocity vectors $\{u' \otimes q : u' \in U \}$ and $\{ u\otimes q' : q' \in Q \}$ lie
in $T_{\beta} \widetilde{\mathcal A}$ at a point $\beta = u \otimes q$.
But these vectors do not span the whole tangent space $T_{\beta} \widetilde{\mathcal A}$
since $\{u' \otimes q : u' \in \mathbb C u \} = \{ u\otimes q' : q' \in \mathbb C q \}$.
Therefore, we additionally consider a curve $u \otimes q +c_t u^2$
such that $c_t \in \mathbb C$ and $c_0=0$,
from which we obtain the tangent vectors of the form $c u^2$
for some $c \in \mathbb C$.

The second fundamental form $\sigma\colon T_{\beta} \times T_{\beta} \rightarrow (T_{[E]}X)/T_{\beta}$
is given by the differential of the Gauss map
$\widetilde{\mathcal A} \rightarrow \Gr(d, T_{[E]}X), \,
\beta
%\in \widetilde{\mathcal A}
\mapsto [T_{\beta}\widetilde{\mathcal A}]
%\in \Gr(n, V)
,$
%\begin{eqnarray*}
%\widetilde{\mathcal A} &\rightarrow& \Gr(n, V) \\
%\beta &\mapsto& \,\,\,\, [T_{\beta}\widetilde{\mathcal A}]
%\end{eqnarray*}
where $d=\dim \widetilde{\mathcal A}$, as explained in Section 2.2.

Let $\{u_t\} \subset U$ be a curve with $u_0=u$ and $\{q_t\} \subset
Q$ be a curve with $q_0=q$.
Then the holomorphic curves $[T_{\beta_t}]$ in $\Gr(d,T_{[E]}X)$
for $\{ \beta_t \} \subset \widetilde{\mathcal A}$ such that
$\beta_0 = \beta$ are as follows:
\begin{enumerate}\setlength{\baselineskip}{0.85 \baselineskip}
\item[{\rm (1)}] for $\beta_t=u_t \otimes q +u_t^2$,
$T_{\beta_t}=\{u_t \otimes q' + u' \otimes q + 2 u_t \circ u' : u' \in U, q' \in Q\}$;

\item[{\rm (2)}] for $\beta_t=u \otimes q_t + u^2$,
$T_{\beta_t} = \{ u \otimes q' + u' \otimes q_t + 2 u \circ u' : u' \in U, q' \in Q\}$;

\item[{\rm (3)}] for $\beta_t=u_t \otimes q$, $T_{\beta_t}=\{u_t \otimes q' + u' \otimes q + c u_t^2
% u_t \circ u_t
 : u' \in U, q' \in Q, c \in \mathbb C \}$;

\item[{\rm (4)}] for $\beta_t=u \otimes q_t$, $T_{\beta_t} =\{ u \otimes q' + u' \otimes q_t + c u^2
% u\circ u
 : u' \in U, q' \in Q, c \in \mathbb C \}$;

\item[{\rm (5)}] for $\beta_t = u \otimes q +c_t u^2$, $T_{\beta_t}
%=T_{c_t \gamma_t}
=\{u \otimes q' + u' \otimes q + c_t u \circ u' + c u^2
 : u' \in U, q' \in Q, c \in \mathbb C \}$,
where $\{c_t\} \subset \mathbb C$ is a curve with $c_0=0$.
\end{enumerate}
By differentiating the curve $[T_{\beta_t}]$
%, $[T_{\gamma_t}]$
in $\Gr(d, T_{[E]} X)$, we can compute the second fundamental form $\sigma$
of $\widetilde{\mathcal A}$. % as mentioned in Section 4.1.
To be specific, for any tangent vectors $\xi, \zeta \in T_{\beta}
\widetilde{\mathcal A}$ we choose
\begin{itemize}\setlength{\baselineskip}{0.85 \baselineskip}
  \item a holomorphic curve $\beta_t$ into $\widetilde{\mathcal A}$
  such that $\beta_0 = \beta$ and $\frac{d}{dt}|_{t=0} \beta_t = \xi$,
  which gives the curve $[T_{\beta_t}]$ in $\Gr(d, T_{[E]} X)$,
  \item a vector field $\rho_t$ %on $V$
along the above curve $\beta_t$
  such that $\rho_0=\zeta$ and $\rho_t \in T_{\beta_t}$ for every $t$.
\end{itemize}
Then we have $\sigma(\xi, \zeta) =\sigma(\frac{d}{dt}|_{t=0}\beta_t,
\rho_0)=\frac{d}{dt}|_{t=0}\rho_t$.\\

\noindent {\rm (}{\it Case I : $\beta=u \otimes q + u^2$}{\rm )}. (i) First, to
compute $\sigma(u' \otimes q + 2u \circ u', u \otimes q')$, take a curve
$\beta_t=u_t \otimes q +u_t^2$ as in (1) and assume that $u_0=u$,
$\frac{d}{dt}|_{t=0}u_t=u'$. Then $\beta_0=u \otimes q+u^2=\beta$
and $\frac{d}{dt}|_{t=0} \beta_t=u' \otimes q + 2 u \circ u'$. Since
$u_t \otimes q' \in T_{\beta_t}$ for any $t$, the differential
$\frac{d}{dt}|_{t=0} [T_{\beta_{t}}] \colon T_{\beta} \,\, \rightarrow
\,\, T_{[E]} X/T_{\beta}$ maps $u \otimes q' \in T_{\beta}$ to
$\frac{d}{dt}|_{t=0} u_t \otimes q' =u' \otimes q'$. Thus we have
$\sigma(u' \otimes q + 2u \circ u', u \otimes q')=u' \otimes q'$.\\
(ii) Taking the same curve $\beta_t=u_t \otimes q +u_t^2$ as in (i),
$u'' \otimes q + 2 u_t \circ u'' \in T_{\beta_t}$ for any $t$. So
$\sigma(u' \otimes q + 2 u \circ u',\,\, u'' \otimes q + 2 u \circ
u'') = \frac{d}{dt}|_{t=0}(u'' \otimes q + 2 u_t \circ u'') = 2 u'
\circ u''$.\\
(iii) For $\beta_t=u \otimes q_t + u^2$ with
$\frac{d}{dt}|_{t=0}q_t=q'$ as in (2), $u \otimes q'' \in T_{\beta_t}$
for any $t$.
So $\sigma(u \otimes q', \,\,u \otimes q'')=\frac{d}{dt}|_{t=0} u \otimes q'' =0$. \vspace{1.7mm}

\noindent  {\rm (}{\it Case II : $\beta=u \otimes q $}{\rm )}. (i) Similarly, we
take $\beta_t=u_t \otimes q$ as in (3) and assume that
$u_0=u, \frac{d}{dt}|_{t=0}u_t=u'$. Then $\beta_0=u \otimes q=\beta$,
$\frac{d}{dt}|_{t=0} \beta_t=u' \otimes q$ and $u_t \otimes q' \in
T_{\beta_t}$ for any $t$, hence $\sigma(u' \otimes q , u \otimes q')
= \frac{d}{dt}|_{t=0} u_t \otimes q' = u' \otimes q'$.\\
(ii) For the above curve $\beta_t=u_t \otimes q$, $u_t^2 \in
T_{\beta_t}$ for any $t$. So $\sigma(u' \otimes q, u ^2) =
\frac{d}{dt}|_{t=0} u_t^2 = 2u \circ u'$.\\
(iii) For the above curve $\beta_t=u_t \otimes q$, $u'' \otimes q
\in T_{\beta_t}$ for any $t$.
So $\sigma(u' \otimes q , u''\otimes q) = \frac{d}{dt}|_{t=0} u'' \otimes q = 0$.\\
(iv) Taking $\beta_t=u \otimes q_t$ as in (4), $\beta_0=u \otimes
q=\beta$ and $\frac{d}{dt}|_{t=0} \beta_t = u \otimes q'$. Since
$u\otimes q'' \in T_{\beta_t}$ for any $t$, we obtain
$\sigma(u \otimes q', u \otimes q'')=\frac{d}{dt}|_{t=0} u \otimes q'' =0$.\\
(v) For the above curve $\beta_t=u \otimes q_t$, $u^2 \in
T_{\beta_t}$ for any $t$.
So $\sigma(u \otimes q', u^2) = \frac{d}{dt}|_{t=0} u^2= 0$.\\
(vi) Finally, we take $\beta_t=u \otimes q + c_t u^2$ as in (5) and
assume that $\frac{d}{dt}|_{t=0}c_t=1$.
Then $\beta_0=u \otimes q=\beta$ and $\frac{d}{dt}|_{t=0} \beta_t =
u^2$. Since $u^2 \in T_{\beta_t}$ for any $t$, we obtain
$\sigma(u^2, u^2) = \frac{d}{dt}|_{t=0} u^2 = 0$.
\end{proof}

\begin{definition}
Let $X$ be a polarized uniruled projective manifold equipped with
a minimal rational component $\mathcal K$.
Assume that $Z$ is an embedded submanifold in $X$.
Let $\mathcal A:=\mathcal C_x(X) \subset \mathbb P(T_xX)$ and
$\mathcal B:=\mathcal C_x(Z) \subset \mathbb P(T_xZ) \subset \mathbb P(T_xX)$ be
the varieties of minimal rational tangents at a common general point $x$ of $X$ and $Z$, respectively.
We say that the pair $(\mathcal A, \mathcal B)$ is \emph{nondegenerate} if
  $$\Ker \sigma_{\beta}( \,\cdot \, , T_{ \beta} \widetilde{ \mathcal B} ) = \mathbb C \beta $$
  for any $\beta \in \widetilde{\mathcal B}$,
  where $\sigma_{\beta}\colon
  T_{\beta} \widetilde{\mathcal A} \times T_{\beta} \widetilde{\mathcal A }
  \rightarrow (T_xX)/T_{\beta} \widetilde{\mathcal A}$
  is the second fundamental form of the affine cone $\widetilde{\mathcal A}$ in $T_xX$ at $\beta$.
\end{definition}

\begin{proposition} \label{nondegenerate: symplectic}
Let $X$ be the symplectic Grassmannian $\Gr_{\omega}(k, 2\ell)$ with $1 < k < \ell$ and
$Z$ be an odd symplectic Grassmannian $\Gr_{\omega}(k, 2\ell; F_a, F_{2\ell-a-1})$ with $0 \leq a < k$.
Let $\mathcal A \subset \mathbb P(T_xX)$ and
$\mathcal B \subset \mathbb P(T_xZ) \subset \mathbb P(T_xX)$ be
the varieties of minimal rational tangents at a common general point $x$ of $X$ and $Z$, respectively.
Then the pair $(\mathcal A, \mathcal B)$ is nondegenerate.
\end{proposition}

\begin{proof}
Now take a point $[E] \in Z$ such that $E \cap F_{a+1} = F_a$.
Then the dimension of $F_{2\ell-a-1} \cap E^{\perp}$ is $2\ell-k-1$
and so $F_{2\ell-a-1}/(F_{2\ell-a-1} \cap E^{\perp})$  is isomorphic to $(E/F_a)^*$
From Lemma 4.2 (2) of Hong-Mok \cite{HoM2},
the variety $\mathcal B$ of minimal rational tangent of $Z$ at a general point $[E] \in Z$ is
the projectivization of the affine cone
$$\widetilde{ \mathcal B}=\{ u \otimes q + c u^2 : u \in U_a, q \in Q_a, c \in \mathbb C \}\backslash \{0\}
\subset (U_a \otimes Q_a) \oplus S^2U_a,$$ where $U_a =(E/F_a)^*$ and $Q_a = (F_{2\ell-a-1} \cap E^{\perp})/E$.
Note that $Q_a$ is a codimension 1 subspace of $Q$.
By this description of the varieties of minimal rational tangents of $Z$ and
by the computation of the second fundamental form of $\mathcal A$ (Lemma \ref{second fundamental form}),
we get the desired results.\\
(1) The tangent space $T_{\beta} \widetilde{\mathcal B}$ at $\beta = u \otimes q$ is given by
$\{ u \otimes q' + u' \otimes q + c
u^2 : u' \in U_a, q' \in Q_a, c \in \mathbb C \}$.
Then we have $\Ker \sigma_{\beta}( \,\cdot \, , U_a \otimes q) = U \otimes q$,
$\Ker \sigma_{\beta}( \,\cdot \, , u \otimes Q_a) = \{ u \otimes q' + c u^2 : q' \in Q, c \in \mathbb C \}
%(u \otimes Q) \cup \mathbb C u^2
$ and
$\Ker \sigma_{\beta}(\,\cdot \, , \mathbb C u^2) = \{ u \otimes q' + c u^2 : q' \in Q, c \in \mathbb C \}
%(u \otimes Q) \cup \mathbb C u^2
$.
Therefore, $\Ker \sigma_{\beta}( \,\cdot \, ,
T_{ \beta} \widetilde{\mathcal B} ) = \mathbb C(u\otimes q) = \mathbb C \beta$.
\\
(2) The tangent space $T_{\beta} \widetilde{\mathcal B}$ at $\beta = u \otimes q + u^2$ is given by
$\{ u \otimes q' + u' \otimes q + 2 u \circ u' : u' \in U_a, q' \in Q_a \}$.
Then we have $\Ker \sigma_{\beta}( \,\cdot \, , u \otimes Q_a) = \{ u \otimes q' + c u^2 : q' \in Q, c \in \mathbb C \}
%(u \otimes Q) \cup \mathbb C (u \otimes q + u^2)
$
and $\Ker \sigma_{\beta}( \,\cdot \, , \{u' \otimes q + 2 u \circ u' : u' \in U_a \}) = \mathbb C (u \otimes q + u^2)$.
Therefore, $\Ker \sigma_{\beta}( \,\cdot \, , T_{ \beta} \widetilde{\mathcal B} )
= \mathbb C(u \otimes q + u^2) = \mathbb C \beta$.
\end{proof}

We will use the same notation for $g\in G$ and for the differential of the
action $g\colon X \to X$ at $x\in X$ and its projectivization
$\mathbb P (T_x X) \to \mathbb P (T_{gx} X)$, for simplicity.

\begin{proposition}
\label{projective equivalence:Symplectic Grassmannian}
In the setting of Proposition \ref{nondegenerate: symplectic},
if $\mathcal B'=\mathcal A \cap \mathbb P(W')$ is another linear section of $\mathcal A$
by a linear subspace $\mathbb P(W')$ of $\mathbb P(T_x X)$ such that
$(\mathcal B \subset \mathbb P(T_x Z))$ is projectively equivalent to $(\mathcal B' \subset \mathbb P(W'))$,
then there is an element $h$ in %the reductive part
a Levi factor of the parabolic subgroup $P$ such that $\mathcal B' = h \mathcal B$.
%Let $\mathcal B' = \mathcal A \cap \mathbb P(W')$ be another linear section of $\mathcal A$
%which is projectively equivalent to $\mathcal B \subset \mathbb P(W)$.
%If $Z$ is  of type {\rm(}a{\rm)} {\rm (}respectively, of type {\rm(}b{\rm)}{\rm)}
%in Lemma \ref {VMRT of homogeneous subvarieties},
%then there is an element $h$ in $\SL(U) \times \SL(Q)$
%{\rm(}respectively, in $\SL(U)${\rm)} such that $\mathcal B'=h \mathcal B$.
\end{proposition}

\begin{proof}
Since $\mathcal B$ is a
%$\mathbb P^{2\ell-2k-1}$-bundle on $\mathbb P(U_a)$,
$\mathbb P^{2m-1}$-bundle on $\mathbb P(U_a)$,
$\mathcal B'$ is also a $\mathbb P^{2m-1}$-bundle on
$\mathbb P^r$, where $r=\dim U_a-1=k-a-1$.
Let $\mathcal B_1'$ be %the hyperplane
a codimension 1 linear section of $\mathcal B'$ which is projectively equivalent to %the hyperplane
the codimension 1 linear section $\mathcal B_1:=\mathcal B \cap \mathbb P(U
\otimes Q)\simeq \mathbb P(U_a) \times \mathbb P(Q_a)$ of $\mathcal B$.

Suppose that $\mathcal B_1'$ is not contained in $ \mathbb P(U \otimes Q) $.
Take $b'=u \otimes q +u^2 \in \widetilde{\mathcal B}_1' \cap
(\widetilde{\mathcal A} \backslash  (U \otimes Q))$. Since $\mathcal
B'$ is a linear section $\mathcal A \cap \mathbb P(W')$ of $\mathcal
A$, the tangent space $T_{b'} \widetilde{\mathcal B}'$ at $b'$ is
contained in the intersection $W'\cap T_{b'}\widetilde{\mathcal A}$
and the second fundamental form $\sigma^{\mathcal
B'}_{b'}:T_{b'}\widetilde{\mathcal B}'\times
T_{b'}\widetilde{\mathcal B}' \rightarrow
W'/T_{b'}\widetilde{\mathcal B'}$ of $\widetilde{\mathcal B}'$ at
$b'$, composed with the quotient map $  W'/T_{b'}
\widetilde{\mathcal B}' \rightarrow T_xX/T_{b'}\widetilde{\mathcal A}$,
is  the restriction of the second fundamental form $\sigma_{b'}$ of
$\widetilde{\mathcal A}$ to $T_{b'}\widetilde{\mathcal B}' \times
T_{b'}\widetilde{\mathcal B}'$. % In particular, $\{ v \in T_{b'}\widetilde{\mathcal B}' : \sigma^{\mathcal B'}_{b'}(v,v)=0 \}$ is contained in $$\{ v \in T_{b'} \widetilde{\mathcal A} : \sigma_{b'}(v,v)=0\}=\{ u \otimes q' + c u^2 : q' \in Q, c \in \mathbb C \} %\{u \otimes q' : q' \in Q\}.$
Hence $\{ v \in T_{b'}\widetilde{\mathcal B}' : \sigma^{\mathcal B'}_{b'}(v,v)=0 \}
$ is a linear subspace of $\{u \otimes q' : q' \in Q\}$ because
$\mathcal B'$ is a linear section of $\mathcal A$.
Since $Z$ is not linear, for $b \in \widetilde{\mathcal B}_1$, $\{v
\in T_b\widetilde{\mathcal B} : \sigma^{\mathcal B}_b(v,v)=0 \}$ is
the union of two subspaces
$\{u' \otimes q : u' \in U_a \}$ and $\{u \otimes q' : q' \in Q_a \}$,
while for $b' \in \widetilde{\mathcal B}_1' \cap (\widetilde{\mathcal A} \backslash   (U \otimes Q))$,
$\{ v \in T_{b'}\widetilde{\mathcal B}' : \sigma^{\mathcal B'}_{b'}(v,v)=0 \}$
is only one linear subspace.
Thus the second fundamental form  $\sigma^{\mathcal B}_{b}$ is not isomorphic to
$\sigma^{\mathcal B'}_{b'}$ and
hence $\mathcal B \subset \mathbb P(T_xZ)$
cannot be projectively equivalent to $\mathcal B' \subset \mathbb P(W')$.
Therefore, $\mathcal B_1'$ is contained in
$\mathbb P(U \otimes Q) \cap \mathcal A \simeq \mathbb P(U) \times \mathbb P(Q)$.
By Lemma 2 of Mok \cite{M08} about linear maps between nontrivial tensor product spaces,
any linear section of $\mathbb P(U) \times \mathbb P(Q)$ which is
projectively equivalent to $\mathbb P(U_a) \times \mathbb
P(Q_a) \subset \mathbb P(T_x Z)$ is of the form $\mathbb P(U_a') \times
\mathbb P(Q_a')$ for some subspaces $U_a' \subset U$ and $Q_a' \subset Q$
with $\dim U_a'=\dim U_a$, $\dim Q_a'=\dim Q_a$.
%and $\mathcal B_1'$ is of the form $\mathbb P(U_a') \times \mathbb P(Q)$
%for some subspace $U_a' \subset U$ of dimension $r+1$.

To characterize the variety $\mathcal B$ of
minimal rational tangents of $Z$,
we use the \emph{base locus} $\{v \in T_{\beta}\widetilde{\mathcal A} :
\sigma_{\beta}(v,v)=0 \}$ of the second fundamental form $\sigma$ of
$\widetilde{\mathcal A} \subset T_x X$ at a generic point $\beta \in
\widetilde{\mathcal A}$.
Let $\mathcal B' = \mathcal A \cap \mathbb P(W')$ be a linear
section of $\mathcal A$ which is projectively equivalent to
$\mathcal B \subset \mathbb P(T_x Z)$. Then for a general point $\beta'$
of $\widetilde{\mathcal B}'$ the second fundamental form
$\sigma_{\beta'}^{\mathcal B'}$ of $\widetilde{\mathcal B'}$ at
$\beta'$ is isomorphic to the second fundamental form
$\sigma_{\beta}^{\mathcal B}$ of $\widetilde{\mathcal B}$ at
$\beta $. Hence $\{ v \in T_{\beta'}\widetilde{\mathcal B'} :
\sigma_{\beta'}^{\mathcal B'}(v,v)=0 \}$ is isomorphic to $\{ v \in
T_{\beta }\widetilde{\mathcal B } : \sigma_{\beta }^{\mathcal B
}(v,v)=0 \}$. From the fact that $\mathcal B'$ is a linear section
of $\mathcal A$, it follows that $\{ v \in
T_{\beta'}\widetilde{\mathcal B'} : \sigma_{\beta'}^{\mathcal
B'}(v,v)=0 \}$ is contained in $\{ v \in T_{\beta'
}\widetilde{\mathcal A} : \sigma_{\beta' }(v,v)=0 \}$.
For $b \in \mathcal B' \backslash \mathcal B_1' \subset
\mathcal A \backslash \mathbb P(U \otimes Q)$ the linear space
$\mathbb P^{2m-1}$ in $\mathcal B'$ passing through $b$ is contained in the fiber
of the projection  $\mathcal A \rightarrow \mathbb P(U)$ containing
$b$, because $\{ v \in T_b \widetilde{\mathcal A}: \sigma_b(v,v)=0\}$
is the tangent space to the fiber of $\widetilde{\mathcal A} \rightarrow U$.
Thus $\mathcal B'$ is the restriction of the $\mathbb P^{2m-1}$-bundle on $\mathbb P(U)$
to the subspace $\mathbb P(U_a')$.
Because any hyperplane in $\mathbb P(Q)$ can be transformed another hyperplane in $\mathbb P(Q)$ under the action of $\Sp(Q)$,
$\mathcal B'=h \mathcal B$ for some $h \in \SL(U) \times
\Sp(Q)$ which is the semisimple part of $P$.
\end{proof}

\noindent {\it Proof of Theorem \ref{Main theorem}
in the case that $X$ is the symplectic Grassmannian $\Gr_{\omega}(k, 2\ell)$}.
From Lemma \ref{smooth Schubert variety: symplectic} and Theorem \ref{homogeneous submanifolds},
it suffices to consider odd symplectic Grassmannians %$\Gr_{\omega}(k, 2\ell; F_a, F_{2\ell-a-1})$
in the symplectic Grassmannian $\Gr_{\omega}(k, 2\ell)$ with $1 <k< \ell$.
Let $Z$ be an odd symplectic Grassmannian $\Gr_{\omega}(k, 2\ell; F_a, F_{2\ell-a-1})$ with $0 \leq a < k$.

Let $f\colon U \rightarrow X$ be a holomorphic embedding
from a connected open subset $U$ of $Z$ into $X$
which respects varieties of minimal rational tangents for a general point $z\in U$.
Then $df(\mathcal C_z(Z))$ is the linear section
$\mathcal C_{f(z)}(X) \cap df(\mathbb P(T_zZ))$ of $\mathcal C_{f(z)}(X)$ and
$df(\mathcal C_z(Z)) \subset df(\mathbb P(T_zZ))$ is projectively equivalent to
$\mathcal C_z(Z) \subset \mathbb P(T_zZ)$.
By Proposition \ref{projective equivalence:Symplectic Grassmannian},
for each general point $z \in U$ there is an element $h=h(z)$ in a Levi factor of the parabolic subgroup $P$
%$g=g(z)$ in the reductive part of the parabolic subgroup $P$
such that $df(\mathcal C_z(Z))= \mathcal C_{f(z)}(hZ)$.
Thus $f$ is nondegenerate with respect to $(\mathcal K, \mathcal H)$ by Proposition \ref{nondegenerate: symplectic}.
Then Proposition 2.1 of Hong-Mok \cite{HoM1} implies that
%$f$ preserves the tautological foliations, that is,
$f$ sends minimal rational curves in $Z$ to minimal rational curves in $X$
and we get a rational extension $F\colon Z\rightarrow X$ of $f$
by Proposition \ref{Non-equidimensional Cartan-Fubini type extension theorem}.
%by Theorem 1.1 of Hong-Mok \cite{HoM1}.
Then the total transformation $F(Z)$ of $F$ is \emph{rationally saturated},
i.e., for every smooth point $x \in F(Z)$ and for any minimal rational curve $C$ on $X$ passing through $x$,
$C$ must lie on $F(Z)$ whenever $C$ is tangent to $F(Z)$ at $x$.
For a general point $x$ in $F(Z)$,
the variety $\mathcal C_x(F(Z))$ of minimal rational tangents of $F(Z)$ is
$dF(\mathcal C_z(Z))$ where $x=F(z)$.

Fix a general point $x_0\in U$.
From the homogeneity of $X$ and Proposition \ref{projective equivalence:Symplectic Grassmannian},
%$F(x_0)=x_0$ and $ \mathcal C_{x_0}(F(Z)) =\mathcal C_{x_0}(Z)$ up to the action of $G$.
$F(x_0)=gx_0$ for some $g \in G$ and $\mathcal C_{gx_0}(F(Z)) =hg\mathcal C_{x_0}(Z)$ for an element $h \in P$.
Then $F(\Sigma) =\Sigma$ up to the action of $G$, where $\Sigma$ denotes
the subvariety of $Z$ swept out by minimal rational curves in $Z$
passing through $x_0$. Let $C$ be a standard minimal rational curve in $Z$ passing through $x_0$ and
let $y\in C$ be a smooth point different from $x_0$.
Then the tangent direction $[T_y C]$ is
contained both in $\mathcal C_y(Z)$ and in $\mathcal C_y(F(Z))$.
By the deformation theory of minimal rational curves (Lemma 2.8 of Hong-Mok \cite{HoM1}),
the tangent space $T_y \Sigma$ of $\Sigma$ at $x$ can be
identified with the tangent space of $\widetilde{\mathcal C}_y(Z)$
at $\alpha \in T_x C$.
Note that by Proposition 4.3 of Hong-Mok \cite{HoM2},
if $h$ is an element in the isotropy subgroup $P_{[E]}$ of $G$
such that $h \mathcal B$ and $\mathcal B$ are tangent at a point of intersection,
then $h\mathcal B$ is equal to $\mathcal B$.
Since $F(\Sigma) =\Sigma$, $\mathcal C_y(Z)$
is tangent to $\mathcal C_y(F(Z))$ at $[T_y C] $, and thus we have
$\mathcal C_y(F(Z))=\mathcal C_y(Z)$.
Therefore, $\mathcal C_y(F(Z))=\mathcal C_y(Z)$ for a generic point $y \in \Sigma$.

Since $Z$ is a uniruled projective manifold of Picard number 1,
there is a sequence of irreducible varieties $\mathcal U_0=\{ x_0\}
\subset \mathcal U_1 \subset \cdot\cdot\cdot \subset \mathcal U_k$
with ${\rm dim}\,\mathcal U_k= {\rm dim}\, Z$ such that a general
point in $\mathcal U_{i+1}$ can be connected to a point in $\mathcal
U_i$ by a minimal rational curve in $Z$
(see Section 4.3 of Hwang-Mok \cite{HwM98} or Section 3 of Mok \cite{Mk08}).
Applying the same arguments
as above inductively, we get that $F(\mathcal U_k)=\mathcal U_k$ and thus we have $F(Z)=Z$.
From Proposition \ref{automorphism groups of horosphercal},
we know that $\Aut(Z)$ is isomorphic to the projective odd symplectic group
$\mbox{PSp}(2\ell-1):=((\Sp(2\ell-2) \times \mathbb C^*)/\{\pm 1\}) \ltimes \mathbb C^{2\ell-2}$.
Consequently, Proposition 3.3 of Mihai \cite{Mihai} implies that there exists $g' \in \Sp(2\ell)$ such that $g' |_Z = F$.
%Consequently, $F$ is the identity map up to the action of $G$.
Therefore, $f$ is the restriction of the standard embedding of $Z$ into $X$.
\qed

\section{Smooth Schubert varieties in $F_4$-homogeneous manifolds}

Let us start with the facts about the complex simple Lie algebra $\mathfrak g$ of type $F_4$.
We choose a system $\{\alpha_1, \alpha_2, \alpha_3, \alpha_4\}$ of simple roots
such that $\alpha_3$ and $\alpha_4$ are short roots.
Then the highest long root of $\mathfrak g$ is $2\alpha_1 + 3\alpha_2 + 4\alpha_3 +
2\alpha_4$, hence the grading on $\mathfrak
g$ associated to $\alpha_3$ is of depth 4.
\begin{center}
\begin{picture} (150, 40)
 \put(0,20){$(F_4)$}

 \put(50,20){\circle{5}} %1
 \put(80,20){\circle{5}} %2
 \put(110,20){\circle{5}} %3
 \put(140,20){\circle{5}} %4

 \put(52.5,20){\line(1,0){25}} %1 edges
 \put(82.5,19){\line(1,0){25}} %2
 \put(82.5,21){\line(1,0){25}} %2
 \put(112.5,20){\line(1,0){25}} %3

 \put(93 ,17.5){$>$}

 \put(45,10){$\alpha_1$} %1
 \put(75,10){$\alpha_2$} %2
 \put(105,10){$\alpha_3$} %3
 \put(135,10){$\alpha_4$} %4

 \put(48,30){2} %1
 \put(78,30){3} %2
 \put(108,30){4} %3
 \put(138,30){2} %4
 \end{picture}
\end{center}

Let $\mathfrak p$ be the maximal parabolic subalgebra of $\mathfrak
g$ associated to the simple root $\alpha_3$.
Given an integer $k$,
$-4 \leq k \leq 4$, $\Phi_k$ denotes the set of all roots
$\alpha=\sum_{q=1}^{4} c_q \alpha_q$
with the third coefficient $c_3 = k$. Define
$$\mathfrak g_0 = \mathfrak h \oplus \bigoplus_{\alpha\in\Phi_0}
\mathfrak g_{\alpha},\qquad \mathfrak g_k =
\bigoplus_{\alpha\in\Phi_k} \mathfrak g_{\alpha}, \,\, k \neq 0.$$
Then the parabolic subalgebra $\mathfrak p$ is decomposed as a graded Lie algebra
$\mathfrak p =\mathfrak g_0 \oplus \mathfrak g_{1}\oplus \mathfrak
g_{2}\oplus \mathfrak g_{3} \oplus \mathfrak g_{4}$ with
$$\dim\mathfrak g_0 = 12, \, \dim \mathfrak g_1 = 6,\, \dim \mathfrak g_2 =
9,\, \dim \mathfrak g_3 = 2,\, \dim \mathfrak g_4 = 3.$$

Let $X$ be the rational homogeneous manifold $(F_4, \alpha_3)$ associated to the short root $\alpha_3$.
Then $X$ is the closed $F_4$-orbit of the space of lines on the rational homogeneous manifold of type $(F_4, \alpha_4)$,
which is a smooth hyperplane section of the (complex) Cayley plane $\mathbb{OP}^2=(E_6, \alpha_1)$
(cf. Section 6 of Landsberg-Manivel \cite{LM}).

Since $\dim \mathfrak g = 52$ and $\dim \mathfrak p = 32$,
the rational homogeneous manifold $X$ of type $(F_4, \alpha_3)$ is a projective variety of dimension 20.
Let $o = eP$ be the base point of $X=G/P$.
The tangent space $T_o(G/P)$ is canonically isomorphic to
$\mathfrak g / \mathfrak p$ and so the first Chern number of the tangent
bundle $T_{G/P}$ is computed by
$\displaystyle\sum_{\beta\in\Phi_1\cup\cdots\cup\Phi_4} \beta(H_{\alpha_3})$, where $H_{\alpha_3}$ is the coroot of $\alpha_3$, from
the proof of Proposition 1 in Hwang-Mok \cite{HM02}.
Because
$$\sum_{\beta\in\Phi_1} \beta(H_{\alpha_3})+\sum_{\beta\in\Phi_2}
\beta(H_{\alpha_3}) +\sum_{\beta\in\Phi_3}
\beta(H_{\alpha_3})+\sum_{\beta\in\Phi_4}
\beta(H_{\alpha_3})=1+3+1+2=7,$$
the first Chern class of $X$ is $c_1(X)=7L \in H^2(X, \mathbb Z) \cong H^1(X, \mathcal O_X^*)=\Pic(X)$.
Here $L$ is the ample generator with $\Pic(G/P)=\mathbb Z L$ and
gives an embedding $G/P\subset \mathbb P^{272}=\mathbb P(V(\omega_3))$,
where $\omega_3$ is the third fundamental weight of $G$.
Furthermore, $G/P$ is covered by lines of $\mathbb P^{272}$ and
the Chow space $\mathcal K_o$ consists of all lines passing through $o$, which are contained in $G/P$.
Hence the tangent map $\tau_o\colon \mathcal K_o \to\mathcal C_o$ is an embedding
and the variety $\mathcal C_o$ of minimal rational tangents at $o$ is 5-dimensional, because $c_1(G/P)=7L$.

Now we take a choice of a Levi factor $\mathfrak k$ of $\mathfrak p$.
The semisimple part of $\mathfrak k$ is isomorphic to $\mathfrak{sl}(3, \mathbb C)\oplus\mathfrak{sl}(2, \mathbb C)$.
So a reductive subgroup $K\subset P$ with Lie algebra $\mathfrak k$
is isogenous to $\mathbb C^* \times \SL(3, \mathbb C)\times \SL(2, \mathbb C)$
as a complex Lie group.
Under the identification $T_o(G/P)=\mathfrak g / \mathfrak p$,
$\mathfrak g_{-1}\oplus \mathfrak g_{-2}\oplus \mathfrak g_{-3}
\oplus \mathfrak g_{-4}$ is the graded decomposition into irreducible $K$-modules.

Let $E$ be a 3-dimensional complex vector space with dual $E^*$ with respect to the standard inner product on $E$ and $Q$ be a 2-dimensional complex vector space.
%Let us consider the standard inner products on $E$ and $Q$ which related to $\SO(E)$ and $\SO(Q)$ respectivley, which give us its dual $E^*$ and $Q^*$. Then there is a induced pairing $\langle , \rangle \colon E^* \times E \rightarrow \mathbb C$, where $\langle e^* , f \rangle$ is the evaluation of $e^* \in E^*$ at $f \in E$.
Then, we can check the following $K$-module isomorphisms:
$$\mathfrak g_{-1}=E^* \otimes Q, \quad
\mathfrak g_{-2}=\wedge^2 E^* \otimes S^2 Q, \quad \mathfrak
g_{-3}=Q,\quad \mathfrak g_{-4}=E^*.$$
%where $\wedge^k V$, $\Sym^k V$ denote the $k$-th exterior power and
% $k$-th symmetric power of a representation $V$, respectively.
In particular, one can determine
the highest weight variety $\mathcal W_1 %\mathbb PW
\subset \mathbb P \mathfrak
g_{-1}$ consisting of highest weight vectors of the irreducible
$K$-module $\mathfrak g_{-1}$.
Because the highest weight variety $\mathcal W_1 \subset \mathbb P
\mathfrak g_{-1}$ of $X$ is a homogeneous manifold associated to the
marked Dynkin diagram having markings corresponding to the simple
roots $\alpha_2$ and $\alpha_4$ which are adjacent to $\alpha_3$ in
the Dynkin diagram of the semisimple part of $P$, we have
$\mathcal W_1=\mathbb P^2\times\mathbb P^1\subset\mathbb P^5$ embedded in the
Segre embedding and its affine cone is equal to
$\{e^* \otimes q \in E^*
\otimes Q : e^* \in E^*, q \in Q\} \backslash \{0\}$.

\begin{picture} (350, 40)
 \put(-15,20){$(F_4, \alpha_3)$}

 \put(40,20){\circle{5}} %1
 \put(70,20){\circle{5}} %2
 \put(95.8,18){$\times$} %3
 \put(131,20){\circle{5}} %4

 \put(42.5,20){\line(1,0){25}} %1 edges
 \put(72.5,19){\line(1,0){25}} %2
 \put(72.5,21){\line(1,0){25}} %2
 \put(102.5,20){\line(1,0){25.5}} %3

 \put(83,17.5){$>$} %4 arrow

 \put(145,18){$\longrightarrow$} %arrow

 \put(184,20){\circle{5}} %1
 \put(210,18){$\times$} %2
% \put(110,20){\circle{5}} %3
 \put(270,18){$\times$} %4

 \put(186.5,20){\line(1,0){25}}

 \put(300,20){$\mathcal W_1 \subset \mathbb P \mathfrak g_{-1}$}
 \end{picture}

The variety $\mathcal C_o(X)$ of minimal rational tangents at the base point $o \in X$
contains the highest weight variety $\mathcal W_1$ in $\mathbb P \mathfrak g_{-1}$
but $\mathcal C_o(X)$ is strictly bigger than $\mathcal W_1$ since $\dim \mathcal C_o(X)=5$.
Hence, we consider the highest weight variety $\mathcal W_2 \subset \mathbb P \mathfrak g_{-2}$ with respect to the $K$-action, which we expect to be contained in $\mathcal C_o(X)$.
The affine cone of $\mathcal W_2$ is equal to % contained in
%given by
%$$\widetilde{\mathcal W_2}=
$$\{(e_1^* \wedge e_2^*) \otimes q^2 \in \wedge^2 E^* \otimes S^2 Q :
  e_1^*, e_2^* \in E^*, q \in Q\} \backslash \{0\}.$$
%And $\mathcal C_o$ is contained in $\mathbb P(\mathfrak g_{-1}\oplus\mathfrak g_{-2})$.
By Section 3 of Hwang-Mok \cite{HwM04b} or Proposition 6.9 of Landsberg-Manivel \cite{LM},
%the variety $\mathcal A$ of minimal rational tangents of $X$ at $o \in X$ is
$\mathcal C_o(X)$ is contained in $\mathbb P(\mathfrak g_{-1}\oplus\mathfrak g_{-2})$
and is the projectivization of the affine cone
$$\{e^* \otimes q + (e_1^* \wedge e_2^*) \otimes q^2 :
  e^* \wedge e_1^* \wedge e_2^* = 0 , \, e^*, e_1^*, e_2^* \in E^*, q \in Q\} \backslash \{0\}$$
in $(E^* \otimes Q) \oplus (\wedge^2 E^* \otimes S^2 Q)$.
%where $E^*$ is a complex vector space of dimension 3 and $Q$ is a complex vector space of dimension 2.
Since $\wedge^2 E^*$ is isomorphic to $E$ as $\SL(E)$-modules, we will
make a fixed choice of the identification.
Now, we denote a subvariety $\mathcal C_o(X) \subset \mathbb P((E^* \otimes Q) \oplus (E \otimes S^2 Q))$ by $\mathcal A$.
Then the affine cone
$\widetilde{\mathcal A}$ over $\mathcal A$ is given by
\begin{eqnarray*}
\widetilde{\mathcal A}=\{e^* \otimes q + f \otimes q^2 :\,\, \langle
e^*,f \rangle =0 ,\,\, e^* \in E^*, f \in E, q \in Q\} \backslash \{0\},
\end{eqnarray*}
where $\langle e^*, f \rangle$ denotes the evaluation of $e^*$ at $f$.
Under the projection map $e^* \otimes q + %(e_1^* \wedge e_2^*)
f \otimes q^2 \mapsto q$,
$\mathcal A$ is a fiber bundle over $\mathbb P(Q)=\mathbb P^1$ with
fibers which are isomorphic to a $4$-dimensional quadric $\mathbb Q^4$.
In other words, $\mathcal A$ is the Grassmannian bundle of 2-planes of
the vector bundle $\mathcal E^*$ on $\mathbb P^1$,
where $\mathcal E$ is a vector bundle of rank 4 which splits as $\mathcal O(1)^3 \oplus \mathcal O$.
In fact, the Pl\"ucker line bundle $\xi$ on $\Gr(2, \mathcal E^*)$ defines
an embedding of $\Gr(2, \mathcal E^*)$ into $\mathbb P H^0(\Gr(2, \mathcal E^*), \xi)$.
Since $H^0(\Gr(2, \mathcal E^*), \xi)=H^0(\mathbb P^1, \wedge^2 \mathcal E)=H^0(\mathbb P^1, \mathcal O(1)^3 \oplus \mathcal O(2)^3)$,
under the identification $\mathbb P^1=\mathbb P(Q^*)$, we have $H^0(\Gr(2, \mathcal E^*), \xi)=(E^* \otimes Q) \oplus (\wedge^2 E^* \otimes S^2 Q)$.
For the detailed descriptions as a projective variety, see Section 2 of Hwang-Mok \cite{HwM04b}.

\begin{lemma}\label{second fundamental form:F4}
Let $X$ be the rational homogeneous manifold of type $(F_4, \alpha_3)$
and $\mathcal A$ be the variety of minimal rational tangents of $X$
at a point $x\in X$. The tangent space $T_{\beta}$ of
$\widetilde{\mathcal A}$ at $\beta \in \widetilde{\mathcal A}$ is
given by
\begin{eqnarray*}
T_{\beta} &=&\{ e^* \otimes q' + e'^* \otimes q +f' \otimes q^2 + f \otimes (2q\circ q') \\
 & & :\,\, \langle e'^*,f \rangle + \langle e^*,f' \rangle =0 ,\,\, e'^* \in E^*, f' \in E, q' \in Q \}
 \text{ if } \beta = e^* \otimes q + f \otimes q^2, \\
T_{\beta} &=& \{ e^* \otimes q' + e'^* \otimes q +f' \otimes q^2 :
\,\, \langle e^*,f' \rangle =0 ,\,\, e'^* \in E^*, f' \in E, q' \in
Q \} \\ & & \text{ if } \beta = e^* \otimes q.
\end{eqnarray*}
The second fundamental form $\sigma \colon T_{\beta} \times
T_{\beta} \longrightarrow (T_xX)/T_{\beta} $ of $\widetilde{\mathcal A}
\subset T_xX$ at $\beta \in \widetilde{\mathcal A}$ is given as
follows:
\begin{enumerate}

\item[(I)] for $\beta=e^* \otimes q + f \otimes q^2$,
\begin{eqnarray*}
&&\sigma(e^* \otimes q' + f \otimes (2q \circ q') ,\,\, e^* \otimes
q'' + f \otimes (2q \circ q''))
= f \otimes (2q' \circ q'')\\
&&\sigma(e^* \otimes q' + f \otimes (2q \circ q') ,\,\, e'^* \otimes q)= e'^* \otimes q' \\
&&\sigma(e^* \otimes q' + f \otimes (2q \circ q') ,\,\, f' \otimes q^2) = f' \otimes (2q \circ q') \\
&&\sigma(e'^* \otimes q ,\,\, e''^* \otimes q) = 0\\
&&\sigma(e'^* \otimes q ,\,\, f' \otimes q^2) = (-\langle e'^*,f' \rangle e)\otimes q^2 \\
&&\sigma(f' \otimes q^2, \,\, f'' \otimes q^2)=0{\rm;}
\end{eqnarray*}

\item[{(II)}] for $\beta=e^* \otimes q$,
\begin{eqnarray*}
&&\sigma(e^* \otimes q' ,\,\, e^* \otimes q'')= 0\\
&&\sigma(e^* \otimes q' ,\,\, e'^* \otimes q) = e'^* \otimes q' \\
&&\sigma(e^* \otimes q' ,\,\, f' \otimes q^2) = f' \otimes (2q \circ q') \\
&&\sigma(e'^* \otimes q ,\,\, e''^* \otimes q) = 0\\
&&\sigma(e'^* \otimes q ,\,\, f' \otimes q^2) = (-\langle e'^*,f' \rangle e)\otimes q^2 \\
&&\sigma(f' \otimes q^2, \,\, f'' \otimes q^2)=0,
\end{eqnarray*}

\end{enumerate}
where $e'^*, e''^* \in E^*, f', f'' \in E$ and $q', q'' \in Q$.
\end{lemma}

\begin{proof}
This is given in Lemma 4.2 of Hong-Park \cite{HP} without details.
We give the details of the proof.
First, to obtain the tangent space $T_{\beta} \widetilde{\mathcal A}$,
we consider the velocity vectors of curves in the affine cone
$\widetilde{\mathcal A}$. Let $\{e^*_t\} \subset E^*$,
$\{f_t\} \subset E$ and $\{q_t\} \subset Q$ be curves with $e^*_0=e^*$, $f_0=f$ and $q_0=q$, respectively.
Assuming $\langle e^*,f \rangle =0$, the curve $e^* \otimes q_t + f \otimes q_t^2$
lies in the affine cone $\widetilde{\mathcal A}$ and passes through
a point $e^* \otimes q + f \otimes q^2$.
Since its velocity vector is $e^* \otimes q' + f \otimes (2 q\circ q')$ for some $q' \in Q$,
$e^* \otimes q' + f \otimes (2 q\circ q')\in T_{\beta} \widetilde{\mathcal A}$.
If we take $e^*_t$ and $f_t$ satisfying $\langle e^*_t ,f_t \rangle =0$ and $f_0=f$,
then $e^*_t \otimes q + f_t \otimes q^2$ is a curve passing through a point
$e^* \otimes q + f \otimes q^2$ in $\widetilde{\mathcal A}$
and its velocity vector is $e'^* \otimes q + f' \otimes q^2$
for some $e'^* \in E^*$, $f'\in E$ such that $\langle e'^*,f \rangle + \langle e^*,f' \rangle =0$.
Next, for the curve $\beta_t=e^* \otimes q_t + f_t \otimes q_t^2$
with $f_0=0$ and $\langle e^*, f_t \rangle =0$, $\beta_0=e^* \otimes q$
and $\frac{d}{dt}|_{t=0} \beta_t
= e^* \otimes (\frac{d}{dt}|_{t=0} q_t) + (\frac{d}{dt}|_{t=0} f_t) \otimes q^2 + f_0 \otimes (\frac{d}{dt}|_{t=0} q_t^2)
= e^* \otimes q'+f' \otimes q^2$ for some $f'\in E$, $q' \in Q$
such that $\langle e^*, f' \rangle =0$.

By a similar computation as in Lemma \ref{second fundamental form},
we get the above results. Let $\{e^*_t\} \subset E^*$, $\{f_t\}
\subset E$ and $\{q_t\} \subset Q$ be curves with $e^*_0=e^*$, $f_0=f$ and
$q_0=q$, respectively. Then the holomorphic curves $[T_{\beta_t}]$ in $\Gr(d,T_xX)$
for $\{ \beta_t \} \subset \widetilde{\mathcal A}$ such that
$\beta_0 = \beta$ are as follows:
\begin{enumerate}%\setlength{\baselineskip}{0.85 \baselineskip}
\item[{\rm (1)}] for $\beta_t=e^* \otimes q_t + f \otimes q_t^2$,
$T_{\beta_t}=\{ e^* \otimes q' + e'^* \otimes q_t +f' \otimes q_t^2
+ f \otimes (2q_t \circ q') : \langle e'^*,f \rangle + \langle e^*,
f' \rangle =0 ,\,\, e'^* \in E^*, f' \in E, q' \in Q \}$;

\item[{\rm (2)}] for $\beta_t=e^*_t \otimes q+f \otimes q^2$,
$T_{\beta_t}=\{ e^*_t \otimes q' + e'^* \otimes q + f' \otimes q^2 +
f \otimes (2q \circ q') : \langle e'^*,f \rangle + \langle e^*_t ,
f' \rangle =0 ,\,\, e'^* \in E^*, f' \in E, q' \in Q \}$;

\item[{\rm (3)}] for $\beta_t= e^* \otimes q + f_t \otimes q^2$ with $f_0=f$,
$T_{\beta_t}=\{ e^* \otimes q' + e'^* \otimes q +f' \otimes q^2 +
f_t \otimes (2q\circ q') : \langle e'^*,f_t \rangle  + \langle e^*,
f' \rangle =0 ,\,\, e'^* \in E^*, f' \in E, q' \in Q \}$;

\item[{\rm (4)}] for $\beta_t=e^* \otimes q_t$,
$T_{\beta_t}=\{ e^* \otimes q' + e'^* \otimes q_t +f' \otimes q_t^2
: \langle e^*, f' \rangle =0 ,\,\, e'^* \in E^*, f' \in E, q' \in Q
\}$;

\item[{\rm (5)}] for $\beta_t=e^*_t \otimes q$,
$T_{\beta_t}=\{ e^*_t \otimes q' + e'^* \otimes q + f' \otimes q^2 :
\langle e^*_t , f' \rangle =0 ,\,\, e'^* \in E^*, f' \in E, q' \in Q
\}$;

\item[{\rm (6)}] for $\beta_t= e^* \otimes q + f_t \otimes q^2$ with $f_0=0$,
$T_{\beta_t}=\{ e^* \otimes q' + e'^* \otimes q +f' \otimes q^2 +
f_t \otimes (2q\circ q') : \langle e'^*,f_t \rangle  + \langle e^*,
f' \rangle =0 ,\,\, e'^* \in E^*, f' \in E, q' \in Q \}$.
\end{enumerate}

As in Lemma \ref{second fundamental form}, the second fundamental
form is computed in the following manner :
 $\sigma(\frac{d}{dt}|_{t=0}\beta_t,\rho_0)=\frac{d}{dt}|_{t=0}\rho_t$,
 where $\rho_t$ is a vector field %on $V$
along the curve $\beta_t$
 such that $\rho_t \in T_{\beta_t}$ for every $t$.\\

\noindent {\rm (}{\it Case I : $\beta=e^* \otimes q + f \otimes q^2$}{\rm )}.
(i) We take a curve $\beta_t=e^* \otimes q_t + f \otimes q_t^2$ as
in (1) and assume that $\frac{d}{dt}|_{t=0}q_t=q'$. Then $\beta_0=e^*
\otimes q + f \otimes q^2=\beta$ and $\frac{d}{dt}|_{t=0}
\beta_t=e^* \otimes q' + f \otimes (2q \circ q')$. Since $e^*
\otimes q'' + f \otimes (2q_t \circ q'')\in T_{\beta_t}$ for any
$t$, the differential $\frac{d}{dt}|_{t=0} [T_{\beta_{t}}] \colon
T_{\beta} \,\, \rightarrow \,\, V/T_{\beta}$ maps $e^* \otimes q'' +
f \otimes (2q \circ q'') \in T_{\beta}$ to $\frac{d}{dt}|_{t=0} (e^*
\otimes q'' + f \otimes (2q_t \circ q'')) = f \otimes
(2(\frac{d}{dt}|_{t=0} q_t) \circ q'')= f \otimes (2q' \circ q'')$.
Thus we have $\sigma(e^* \otimes q' + f \otimes (2q \circ q') ,\,\,
e^* \otimes q'' + f \otimes (2q \circ q'')) = f \otimes (2q' \circ
q'')$.\\
(ii) If $e'^* \otimes q \in T_{\beta}$, then the relation $\langle
e'^*,f \rangle =0$ holds. For the above curve $\beta_t=e^* \otimes
q_t + f \otimes q_t^2$, $e'^* \otimes q_t \in T_{\beta_t}$ for any
$t$. So $\sigma(e^* \otimes q' + f \otimes (2q \circ q') ,\,\, e'^*
\otimes q) = \frac{d}{dt}|_{t=0}(e'^* \otimes q_t) = e'^* \otimes
q'$.\\
(iii) If $f' \otimes q^2 \in T_{\beta}$, then the relation $\langle
e^*,f' \rangle =0$ holds. For the above curve $\beta_t=e^* \otimes
q_t + f \otimes q_t^2$, $f' \otimes q_t^2 \in T_{\beta_t}$ for any
$t$. So $\sigma(e^* \otimes q' + f \otimes (2q \circ q') ,\,\, f'
\otimes q^2) = \frac{d}{dt}|_{t=0}(f' \otimes q_t^2) = f' \otimes (2q \circ q')$.\\
(iv) Taking a curve $\beta_t=e^*_t \otimes q+f \otimes q^2$ as in (2)
such that $\frac{d}{dt}|_{t=0}e^*_t=e'^*$, $\beta_0=e^* \otimes q+f
\otimes q^2=\beta$ and $\frac{d}{dt}|_{t=0} \beta_t = e'^* \otimes
q$. If $e''^* \otimes q \in T_{\beta}$, then the relation $\langle
e''^*,f \rangle =0$ holds. Since $e''^* \otimes q \in T_{\beta_t}$
for any $t$, we obtain $\sigma(e'^* \otimes q ,\,\, e''^* \otimes q)
=\frac{d}{dt}|_{t=0} e''^* \otimes q = 0$.\\
(v) We take the above curve $\beta_t=e^*_t \otimes q+f \otimes q^2$
and a vector field $f_t \otimes q^2$ along $\beta_t$ with $f_0=f'$.
Then $f_t \otimes q^2 \in T_{\beta_t}$ whenever $\langle e_t^*, f_t
\rangle=0$ for any $t$. Differentiating the equation $\langle e_t^*,
f_t \rangle=0$ at $t=0$, we have $\langle e'^*, f' \rangle + \langle
e^*, (\frac{d}{dt}|_{t=0}f_t) \rangle=0$. Hence
$\frac{d}{dt}|_{t=0}f_t = -\langle e'^*, f' \rangle e$ and so
$\sigma(e'^* \otimes q ,\,\, f' \otimes q^2) = \frac{d}{dt}|_{t=0}
f_t \otimes q^2
= (-\langle e'^*,f' \rangle e)\otimes q^2$.\\
(vi) Taking a curve $\beta_t=e^* \otimes q + f_t \otimes q^2$ as in (3)
such that $f_0=f$ and $\frac{d}{dt}|_{t=0} f_t=f'$,
$\beta_0=e^* \otimes q+f \otimes q^2=\beta$ and
$\frac{d}{dt}|_{t=0} \beta_t = f' \otimes q^2$.
Since $f'' \otimes q^2 \in T_{\beta_t}$ for any $t$, we obtain
$\sigma(f' \otimes q^2, \,\, f'' \otimes q^2)=\frac{d}{dt}|_{t=0} f'' \otimes q^2=0$. \vspace{1.7mm}

\noindent  {\rm (}{\it Case II : $\beta=e^* \otimes q$}{\rm )}. (i) Now take a
curve $\beta_t=e^* \otimes q_t$ as in (4) and assume that
$\frac{d}{dt}|_{t=0}q_t=q'$. Then $\beta_0=e^* \otimes q=\beta$ and
$\frac{d}{dt}|_{t=0} \beta_t=e^* \otimes q'$. Since $e^* \otimes q''
\in T_{\beta_t}$ for any $t$, we have $\sigma(e^* \otimes q' ,\,\,
e^* \otimes q'')=\frac{d}{dt}|_{t=0} e^* \otimes q'' = 0$.\\
(ii) For the above curve $\beta_t=e^* \otimes q_t$, $e'^* \otimes
q_t \in T_{\beta_t}$ for any $t$. So $\sigma(e^* \otimes q' ,\,\,
e'^* \otimes q) = \frac{d}{dt}|_{t=0}(e'^* \otimes q_t) = e'^* \otimes q'$.\\
(iii) If $f' \otimes q^2 \in T_{\beta}$, then the relation $\langle
e^*,f' \rangle =0$ holds. For the above curve $\beta_t=e^* \otimes
q_t$, $f' \otimes q_t^2 \in T_{\beta_t}$ for any $t$. So $\sigma(e^*
\otimes q' ,\,\, f' \otimes q^2) = \frac{d}{dt}|_{t=0}(f' \otimes
q_t^2)
= f' \otimes (2q \circ q')$.\\
(iv) Taking a curve $\beta_t=e^*_t \otimes q$ as in (5) such that
$\frac{d}{dt}|_{t=0}e^*_t=e'^*$, $\beta_0=e^* \otimes q=\beta$ and
$\frac{d}{dt}|_{t=0} \beta_t = e'^* \otimes q$. Since $e''^* \otimes
q \in T_{\beta_t}$ for any $t$, we obtain $\sigma(e'^* \otimes q
,\,\, e''^* \otimes q)
=\frac{d}{dt}|_{t=0} e''^* \otimes q = 0$.\\
(v) We take the above curve $\beta_t=e^*_t \otimes q$ and a vector
field $f_t \otimes q^2$ along $\beta_t$ with $f_0=f'$. Then $f_t
\otimes q^2 \in T_{\beta_t}$ whenever $\langle e_t^*, f_t \rangle=0$
for any $t$. Differentiating this equation, we know
$\frac{d}{dt}|_{t=0}f_t = -\langle e'^*, f' \rangle e$ as in (v) of
Case I.
Hence $\sigma(e'^* \otimes q ,\,\, f' \otimes q^2) =
\frac{d}{dt}|_{t=0} f_t \otimes q^2
= (-\langle e'^*,f' \rangle e)\otimes q^2$.\\
(vi) Taking a curve $\beta_t=e^* \otimes q + f_t \otimes q^2$ as in
(6) such that $f_0=0$ and $\frac{d}{dt}|_{t=0} f_t=f'$, $\beta_0=e^*
\otimes q=\beta$ and $\frac{d}{dt}|_{t=0} \beta_t = f' \otimes q^2$.
Since $f'' \otimes q^2 \in T_{\beta_t}$ for any $t$, we obtain
$\sigma(f' \otimes q^2, \,\, f'' \otimes q^2)=\frac{d}{dt}|_{t=0}
f'' \otimes q^2=0$.
\end{proof}

Hong-Kwon \cite{HK} have classified smooth Schubert varieties in the $F_4$-homogeneous manifold $(F_4, \alpha_3)$.
Thus, for the proof of Theorem \ref{Main theorem}, it suffices to consider the only two cases for $Z$:

\begin{lemma}\label{Schubert varieties:F4}
Let $X$ be the rational homogeneous manifold of type $(F_4, \alpha_3)$.
A nonhomogeneous smooth Schubert variety $Z$ of $X$ is one of the followings:
\begin{enumerate}%\setlength{\baselineskip}{0.85 \baselineskip}
\item[{\rm (1)}] the horospherical variety $(B_3, \alpha_2, \alpha_3)$,
\item[{\rm (2)}] the horospherical variety $(C_2, \alpha_2, \alpha_1)$
which is isomorphic to a smooth Schubert variety $\Gr_{\omega}(2, 6; F_0, F_5)$
in the symplectic Grassmannian $(C_3, \alpha_2)$.
%and called an odd Lagrangian Grassmannian.
\end{enumerate}
\end{lemma}

\begin{remark}
Recall that all nonlinear homogeneous submanifolds
associated to subdiagrams of the marked Dynkin diagram of $(F_4, \alpha_3)$ are
$(B_3, \alpha_3)$ and $(C_3, \alpha_2)$.
As considered in Section 3,
the odd symplectic Grassmannian $(C_2, \alpha_2, \alpha_1)$ is
a unique nonhomogeneous smooth Schubert variety of $(C_3, \alpha_2)$.

 
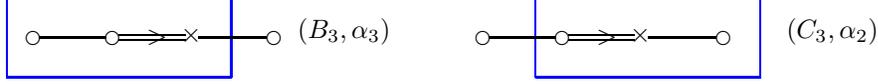
\begin{picture} (150, 40)
 \put(30,20){\circle{5}} %1
 \put(60,20){\circle{5}} %2
 \put(85.8,18.2){$\times$} %3
 \put(121,20){\circle{5}} %4

{\color{blue}
 \put( 20, 5){\line(0,1){30}} %box
 \put( 20,5){\line(1,0){85}}
 \put(105,35){\line( 0,-1){30}}
 \put(105,35){\line(-1,0){85}}}

 \put(32.5,20){\line(1,0){25}} %1 edges
 \put(62.5,19){\line(1,0){25}} %2
 \put(62.5,21){\line(1,0){25}} %2
 \put(92.5,20){\line(1,0){25.5}} %3

 \put(73 ,17.5){$>$} %4 arrow

 \put(130 ,20){$(B_3, \alpha_3)$ }

 \put(200,20){\circle{5}} %1
 \put(230,20){\circle{5}} %2
 \put(255.8,18.2){$\times$} %3
 \put(291,20){\circle{5}} %4

{\color{blue}
 \put( 220, 5){\line(0,1){30}} %box
 \put( 220,5){\line(1,0){85}}
 \put(305,35){\line( 0,-1){30}}
 \put(305,35){\line(-1,0){85}}}

 \put(202.5,20){\line(1,0){25}} %1 edges
 \put(232.5,19){\line(1,0){25}} %2
 \put(232.5,21){\line(1,0){25}} %2
 \put(262.5,20){\line(1,0){25.5}} %3

 \put(243 ,17.5){$>$} %4 arrow

 \put(315 ,20){$(C_3, \alpha_2)$ }
 \end{picture}
\end{remark}

\begin{lemma}\label{VMRT of submanifolds:F4}
Let $Z$ be a nonhomogeneous smooth Schubert variety of the rational homogeneous manifold of type $(F_4, \alpha_3)$.
Then the variety $\mathcal B$ of minimal rational tangents of $Z$ at a general point $z \in Z$ is
\begin{enumerate}%\setlength{\baselineskip}{0.85 \baselineskip}
\item[{\rm (1)}] a $\mathbb P^2$-bundle $\mathbb P(\mathcal O(-1) \oplus \mathcal O(-2)^2)$
over $\mathbb P(Q)=\mathbb P^1$ if $Z$ is of type $(B_3, \alpha_2, \alpha_3)$,
\item[{\rm (2)}] a $\mathbb P^1$-bundle  $\mathbb P(\mathcal O(-1) \oplus \mathcal O(-2))$
over $\mathbb P(Q)=\mathbb P^1$ if $Z$ is of type $(C_2, \alpha_2, \alpha_1)$.
\end{enumerate}
\end{lemma}

\begin{proof} (1) Hong and Kim \cite{HKi} showed that
the variety of minimal rational tangents of the horospherical variety $(B_n, \alpha_{n-1}, \alpha_n)$ is
$\mathbb P(\mathcal O_{\mathbb P^{1}}(-1) \oplus \mathcal O_{\mathbb P^{1}}(-2)^{n-1})$
by calculating the Chern numbers based on a gradation on its tangent space.

(2) As already described in Section 3,
the variety of minimal rational tangents of the odd symplectic Grassmannian $(C_n, \alpha_{k}, \alpha_{k-1})$
is $\mathbb P(\mathcal O_{\mathbb P^{k-1}}(-1)^{2n-2k+1} \oplus \mathcal O_{\mathbb P^{k-1}}(-2))$.
\end{proof}

Let $Z$ be a smooth Schubert variety of type $(B_3, \alpha_2, \alpha_3)$.
The gradation on the tangent space of $Z$ %which
described in Proposition 25 of Kim \cite{Kim} could be embedded in the gradation on the tangent space
$$\mathfrak g_{-1}=E^* \otimes Q, \quad
\mathfrak g_{-2}=E \otimes S^2 Q, \quad \mathfrak
g_{-3}=Q,\quad \mathfrak g_{-4}=E^*$$
as a linear section by Lemma \ref{Schubert varieties:F4} (1) after proper shifting of the gradation on the tangent space of $Z$.
Let $\mathfrak g'_{-1}\oplus \mathfrak g'_{-2} \oplus \mathfrak
g'_{-3}$ be the induced gradation on the tangent space of $Z$ from $X$.
Then
$$\mathfrak g'_{-1}=F^* \otimes Q, \quad
\mathfrak g'_{-2}={F^*}^{\perp} \otimes S^2 Q, \quad \mathfrak
g'_{-3}=\wedge^2 {F^*}^{\perp},$$
where $F^*\subset E^*$ is a fixed subspace of dimension 1 %, which has the highest weight vector under $\SL(E)$-action
and
${F^*}^{\perp} = \{ f\in E : \, \langle e^*, f \rangle = 0, \, \forall e^* \in F^* \}$.
Hence, the variety $\mathcal B$ of minimal rational tangents of $Z$ at a general point $x$ is
 \begin{eqnarray*} \mathcal B&=&\mathbb P(\{ e^* \otimes q
+ f \otimes q^2 : \, e^* \in F^*, \, f \in {F^*}^{\perp}, \, q \in Q
\}) \text{ as a linear section of }\\
\mathcal A&=&\mathbb P(\{e^* \otimes q + f \otimes q^2 :\,\, \langle
e^*,f \rangle =0 ,\,\, e^* \in E^*, f \in E, q \in Q\}).
\end{eqnarray*}
This $\mathcal B$ is a $\mathbb P^2$-bundle
$\mathbb P(\mathcal O_{\mathbb P^1}(-1) \oplus \mathcal O_{\mathbb P^1}(-2)^2)$ over $\mathbb P(Q)=\mathbb P^1$.
This result coincides with Lemma \ref{VMRT of submanifolds:F4} (1).

Let $Z$ be a smooth Schubert variety of type $(C_2, \alpha_2, \alpha_1)$.
By Lemma \ref{Schubert varieties:F4} (2) and Proposition 25 of Kim \cite{Kim},
after proper shifting of the gradation on the tangent space of $Z$,
we let $\mathfrak g'_{-1}\oplus \mathfrak g'_{-2}$ be the induced gradation on the tangent space of $Z$ from $X$. Then
$$\mathfrak g'_{-1}=F^* \otimes Q, \quad
\mathfrak g'_{-2}={F'} \otimes S^2 Q,$$
where $F^*\subset E^*$ is the above fixed subspace and $F'\subset {F^*}^{\perp}$ is an 1-dimensional subspace.
%which has the lowest weight vector under $SL({F^*}^{\perp})$-action.
Hence, the variety $\mathcal B$ of minimal rational tangents of $Z$ at a general point $x$ is
$$\mathcal B=\mathbb P(\{ e^* \otimes q + f \otimes q^2 : \, e^* \in F^*, \, f \in {F'}, \, q \in Q \})$$
as a linear section of $\mathcal A$.
This $\mathcal B$ is
a $\mathbb P^1$-bundle $\mathbb P(\mathcal O_{\mathbb P^1}(-1) \oplus \mathcal O_{\mathbb P^1}(-2))$ over $\mathbb P(Q)=\mathbb P^1$.
This result coincides with Lemma \ref{VMRT of submanifolds:F4} (2).

\begin{proposition} \label{nondegenerate: F_4}
Let $X$ be the rational homogeneous manifold of type $(F_4, \alpha_3)$
and $Z$ be a smooth Schubert variety of type
$(B_3, \alpha_2, \alpha_3)$ or $(C_2, \alpha_2, \alpha_1)$.
Let $\mathcal A \subset \mathbb P(T_xX)$ and
$\mathcal B \subset \mathbb P(T_xZ) \subset \mathbb P(T_xX)$ be
the varieties of minimal rational tangents at a common general point $x$ of $X$ and $Z$.
Then the pair $(\mathcal A, \mathcal B)$ is nondegenerate.
\end{proposition}

\begin{proof}
By this description of the variety $\mathcal B$ of minimal rational tangents of $Z$
as a linear section of the variety $\mathcal A$ of minimal rational tangents of $X$ and
by the computation of the second fundamental form of $\mathcal A$, we get the desired results.

(1) Let $Z$ be a smooth Schubert variety of type $(B_3, \alpha_2, \alpha_3)$.
(i) The tangent space $T_{\beta} \widetilde{\mathcal B}$ at $\beta = e^* \otimes q$ is given by
 $\{ e^* \otimes q' + e'^* \otimes q +f' \otimes q^2 :
\,\, %\langle e^*,f' \rangle =0 ,\,\,
e'^* \in F^*, f' \in {F^*}^{\perp}, q' \in
Q \}$.
Then we have $\Ker \sigma_{\beta}( \,\cdot \, , e^* \otimes Q) = e^* \otimes Q$,
$\Ker \sigma_{\beta}( \,\cdot \, , F^* \otimes q) = %F^* \otimes q
\{ e'^* \otimes q + f' \otimes q^2 : \,\, e'^* \in E^*, f' \in {F^*}^{\perp} \}
$ and
$\Ker \sigma_{\beta}(\,\cdot \, , {F^*}^{\perp} \otimes q^2) = \{ e'^* \otimes q +  f'\otimes q^2 : e'^* \in F^*, f' \in E \}$.
Therefore, $\Ker \sigma_{\beta}( \,\cdot \, , T_{ \beta} \widetilde{\mathcal B} ) = \mathbb C(e^*\otimes q) = \mathbb C \beta$.
(ii) The tangent space $T_{\beta} \widetilde{\mathcal B}$ at $\beta = e^* \otimes q + f\otimes q^2$ is given by
$\{ e^* \otimes q' + e'^* \otimes q +f' \otimes q^2 + f \otimes (2q\circ q') :\,\, e'^* \in F^*, f' \in {F^*}^{\perp}, q' \in Q \}$.
Then we have $\Ker \sigma_{\beta}( \,\cdot \, , F^* \otimes q) \cap \Ker \sigma_{\beta}( \,\cdot \, , {F^*}^{\perp} \otimes q^2)=\{ e'^* \otimes q + f'\otimes q^2 : e'^* \in F^*, f' \in {F^*}^{\perp} \}$
and $\Ker \sigma_{\beta}( \,\cdot \, , \{e^* \otimes q' + 2f\otimes q \circ q' : q' \in Q \}) = \mathbb C (e^* \otimes q + f\otimes q^2)$.
Therefore, $\Ker \sigma_{\beta}( \,\cdot \, , T_{ \beta} \widetilde{\mathcal B} )= \mathbb C(e^* \otimes q + f\otimes q^2) = \mathbb C \beta$.

(2) Let $Z$ be a smooth Schubert variety of type $(C_2, \alpha_2, \alpha_1)$. (i) The tangent space $T_{\beta} \widetilde{\mathcal B}$ at $\beta = e^* \otimes q$ is given by
 $\{ e^* \otimes q' + e'^* \otimes q +f' \otimes q^2 :
\,\, e'^* \in F^*, f' \in {F'}, q' \in
Q \}$.
Then we have $\Ker \sigma_{\beta}( \,\cdot \, , e^* \otimes Q) = e^* \otimes Q$,
$\Ker \sigma_{\beta}( \,\cdot \, , F^* \otimes q) = \{ e'^* \otimes q + f' \otimes q^2 : \,\, e'^* \in E^*, f' \in {F'} \}$ and
$\Ker \sigma_{\beta}(\,\cdot \, , {F'} \otimes q^2) = \{ e'^* \otimes q +  f'\otimes q^2 : e'^* \in F^*, f' \in E \}$.
Therefore, $\Ker \sigma_{\beta}( \,\cdot \, , T_{ \beta} \widetilde{\mathcal B} ) = \mathbb C(e^*\otimes q) = \mathbb C \beta$.
(ii) The tangent space $T_{\beta} \widetilde{\mathcal B}$ at $\beta = e^* \otimes q + f\otimes q^2$ is given by
$\{ e^* \otimes q' + e'^* \otimes q +f' \otimes q^2 + f \otimes (2q\circ q') :\,\,
e'^* \in F^*, f' \in {F'}, q' \in Q \}$.
Then we have $\Ker \sigma_{\beta}( \,\cdot \, , F^* \otimes q) \cap \Ker \sigma_{\beta}( \,\cdot \, , {F'} \otimes q^2)=\{ e'^* \otimes q + f'\otimes q^2 : e'^* \in F^*, f' \in {F^*}^{\perp} \}$
and $\Ker \sigma_{\beta}( \,\cdot \, , \{e^* \otimes q' + 2f\otimes q \circ q' : q' \in Q \}) = \mathbb C (e^* \otimes q + f\otimes q^2)$.
Therefore, $\Ker \sigma_{\beta}( \,\cdot \, , T_{ \beta} \widetilde{\mathcal B} )= \mathbb C(e^* \otimes q + f\otimes q^2) = \mathbb C \beta$.
\end{proof}

\begin{proposition} \label{tangency: F_4}
In the setting of Proposition \ref{nondegenerate: F_4},
if $h$ is an element in the isotropy subgroup $P_{x}$ of $G$ at a general point $x \in Z$
such that $h \mathcal B$ and $\mathcal B$ are tangent at a general point of intersection,
then $h\mathcal B$ is equal to $\mathcal B$.
\end{proposition}

\begin{proof}
We recall $K=\mathbb C^* \times \SL(E^*)\times \SL(Q)$-module isomorphisms:
$$\mathfrak g_{-1}=E^* \otimes Q, \quad
\mathfrak g_{-2}=\wedge^2 E^* \otimes S^2 Q, \quad \mathfrak
g_{-3}=Q,\quad \mathfrak g_{-4}=E^*$$
under the identification $T_oX\cong\mathfrak g_{-1}\oplus \mathfrak g_{-2}\oplus \mathfrak g_{-3}
\oplus \mathfrak g_{-4}$ at the base point $o \in X=G/P$.

(1) Let $Z$ be a smooth Schubert variety of type $(B_3, \alpha_2, \alpha_3)$.
 The variety of minimal rational tangents at a general point $x$ is $$\mathcal B=\mathbb P(\{ e^* \otimes q
+ f \otimes q^2 : \, e^* \in F^*, \, f \in {F^*}^{\perp}, \, q \in Q
\}),$$ where $F^*\subset E^*$ is a subspace of dimension 1 and
${F^*}^{\perp} = \{ f\in E : \, \langle e^*, f \rangle = 0, \, \forall e^* \in F^* \}$.

Let $h$ be an element in the isotropy subgroup $P_x$ of $G$,
then $h=h_0h_1h_2h_3h_4$ where $dh_i \in \bigoplus_{j\in \mathbb Z}\Hom(\frak g_{j},\frak g_{j+i})$.
The left-multiplication actions of $h_i$ at $e^* \otimes q
+ f \otimes q^2 \in \mathcal B$ are
\begin{eqnarray*}
h_0(e^* \otimes q+ f \otimes q^2) &=&h_0e^* \otimes h_0q + h_0 f \otimes h_0 q^2\in (E^* \otimes Q) \oplus (\wedge^2 E^* \otimes S^2 Q)\\
h_1(e^* \otimes q+ f \otimes q^2) &=&e^* \otimes q+ f \otimes q^2+dh_1(f \otimes q^2) \\
h_i(e^* \otimes q+ f \otimes q^2) &=& e^* \otimes q+ f \otimes q^2 \,\text{ for } \,  i=2,3,4.
\end{eqnarray*}

Let $f=f_1^* \wedge f_2^*$ for $f_1^*, f_2^* \in E^*$. For the action $dh_1\in \bigoplus_{j\in \mathbb Z}\Hom(\frak g_{j},\frak g_{j+i})$, we consider the adjoint action of $\frak g_1$ to $\frak g_{-2}$ at $f\otimes q^2$ as follows;
\begin{eqnarray*}
  (E\otimes Q^*) \times (\wedge^2 E^*\otimes \Sym^2 Q) &\rightarrow& E^* \otimes Q \\
  (e' \otimes q'^* \quad, \quad f\otimes q^2) \qquad\quad & \mapsto & c f(e') \otimes q
\end{eqnarray*}
where $e'\in E$, $q'^* \in Q^*$, $c$ is a scalar which is zero if $\langle q'^*,q \rangle=0$,
and $f(e')=(f_1^* \wedge f_2^*)(e')=\langle f_1^*,e' \rangle f_2^* - \langle f_2^*,e' \rangle f_1^*$.
From now on, $h_1$ action on a subspace $F^{*\perp}\subset \wedge^2E^*$ means $f(e')$, i.e., $f(e') \in h_1F^{*\perp}$ and $dh_1(f \otimes q^2)=h_1f\otimes cq$ for some scalar $c$.

Since $dh_1(f \otimes q^2) \in E^* \otimes Q$ and $h_0Q=Q$, it follows that
 $$h\mathcal B=\mathbb P(\{ e^* \otimes q
+ f \otimes q^2 : \, e^* \in h_0F^*+h_0h_1({F^*}^{\perp}), \, f \in h_0({F^*}^{\perp}), \, q \in Q
\}).$$
If $\mathcal B$ and $h\mathcal B$ intersect at a general point
$\beta = e^* \otimes q + f \otimes q^2 \in \mathcal B \cap (h\mathcal B)$,
then the tangent space $T_{\beta} (h\widetilde{\mathcal B})$ is given by
$$\{ e^* \otimes q' + e'^* \otimes q +f' \otimes q^2 + f \otimes (2q\circ q') :\,\, e'^* \in h_0F^*+h_0h_1({F^*}^{\perp}), f' \in h_0.({F^*}^{\perp}), q' \in Q \}.$$
By assumption, $T_{\beta} (h\widetilde{\mathcal B})$ coincide with $T_{\beta} (\widetilde{\mathcal B})$,
we see $h_0F^*+h_0h_1({F^*}^{\perp})=F^*$ and $h_0({F^*}^{\perp})={F^*}^{\perp}$. Hence, $h\mathcal B=\mathcal B$.

(2) Let $Z$ be a smooth Schubert variety of type $(C_2, \alpha_2, \alpha_1)$.
The variety of minimal rational tangents at a general point $x$ is $$\mathcal B=\mathbb P(\{ e^* \otimes q
+ f \otimes q^2 : \, e^* \in F^*, \, f \in {F'}, \, q \in Q
\}),$$ where $F^*\subset E^*$ and
$F'\subset {F^*}^{\perp}$ are subspaces of dimension 1.

If $h$ is an element in the isotropy subgroup $P_x$ of $G$, then
 $$h\mathcal B=\mathbb P(\{ e^* \otimes q
+ f \otimes q^2 : \, e^* \in h_0F^*+h_0h_1F', \, f \in h_0{F'}, \, q \in Q
\}).$$
If $\mathcal B$ and $h\mathcal B$ intersect at a general point
$\beta = e^* \otimes q + f \otimes q^2 \in \mathcal B \cap (h\mathcal B)$,
then the tangent space $T_{\beta} (h\widetilde{\mathcal B})$ is given by
$$\{ e^* \otimes q' + e'^* \otimes q +f' \otimes q^2 + f \otimes (2q\circ q') :\,\, e'^* \in h_0F^*+h_0h_1{F'}, f' \in h_0{F'}, q' \in Q \}.$$
By assumption, $T_{\beta} (h\widetilde{\mathcal B})$ coincides with $T_{\beta} (\widetilde{\mathcal B})$,
we see $h_0F^*+h_0h_1{F'}=F^*$ and $h_0({F^*}^{\perp})={F^*}^{\perp}$. Hence, $h\mathcal B=\mathcal B$.
\end{proof}

\begin{remark}
In the proof (1) of Proposition \ref{tangency: F_4}, suppose that $h_0F^* \subset h_0h_1({F^*}^{\perp})$, the dimension of $h_0h_1({F^*}^{\perp})$ is $2$,
and the dimension of $h_0({F^*}^{\perp})$ is $1$, then $h\mathcal B=\mathbb P(\{ e^* \otimes q
+ f \otimes q^2 : \, e^* \in h_0h_1({F^*}^{\perp}), \, f \in h_0({F^*}^{\perp}), \, q \in Q
\})$ which is isomorphic to $\mathbb P(\mathcal O_{\mathbb P^{1}}(-1)^2 \oplus \mathcal O_{\mathbb P^{1}}(-2))$.
In this case, two rank 3 vector bundles $h \mathcal B=\mathbb P(\mathcal O_{\mathbb P^{1}}(-1)^2 \oplus \mathcal O_{\mathbb P^{1}}(-2))$
and $\mathcal B=\mathbb P(\mathcal O_{\mathbb P^{1}}(-1) \oplus \mathcal O_{\mathbb P^{1}}(-2)^2)$ are not tangent at a general point of intersection.
\end{remark}

\begin{proposition}
\label{projective equivalence: F_4}
In the setting of Proposition \ref{nondegenerate: F_4},
if $\mathcal B'=\mathcal A \cap \mathbb P(W')$ is another linear section of $\mathcal A$
by a linear subspace $\mathbb P(W')$ of $\mathbb P(T_xX)$ such that
$(\mathcal B \subset \mathbb P(T_xZ))$ is projectively equivalent to $(\mathcal B' \subset \mathbb P(W'))$,
then there is an element $h$ in % the reductive part
a Levi factor of $P$ such that $\mathcal B' = h \mathcal B$.
\end{proposition}

\begin{proof}
Let $Z$ be a smooth Schubert variety of type $(B_3, \alpha_2, \alpha_3)$.
Since $\mathcal B$ is a $\mathbb P^{2}$-bundle on $\mathbb P(Q) =\mathbb P^1$,
$\mathcal B'$ is also a $\mathbb P^{2}$-bundle on $\mathbb P^1$.
Let $\mathcal B_1'$ be a linear section of $\mathcal B'$
which is projectively equivalent to the linear section $\mathcal B_1:=\mathcal B \cap \mathbb P(E^*\otimes Q)\simeq \mathbb P(F^*) \times \mathbb P(Q)$ of $\mathcal B$.

Suppose that $\mathcal B_1'$ is not contained in $ \mathbb P(E^* \otimes Q) $.
Take $b'=e^* \otimes q + f\otimes q^2 \in \widetilde{\mathcal B}_1' \cap
(\widetilde{\mathcal A} \backslash  (E^* \otimes Q))$. Since $\mathcal
B'$ is a linear section $\mathcal A \cap \mathbb P(W')$ of $\mathcal
A$, the tangent space $T_{b'} \widetilde{\mathcal B}'$ at $b'$ is
contained in the intersection $W'\cap T_{b'}\widetilde{\mathcal A}$
and the second fundamental form $\sigma^{\mathcal
B'}_{b'}:T_{b'}\widetilde{\mathcal B}'\times
T_{b'}\widetilde{\mathcal B}' \rightarrow
W'/T_{b'}\widetilde{\mathcal B'}$ of $\widetilde{\mathcal B}'$ at
$b'$, composed with the quotient map $  W'/T_{b'}
\widetilde{\mathcal B}' \rightarrow T_xX/T_{b'}\widetilde{\mathcal A}$,
is the restriction of the second fundamental form $\sigma_{b'}$ of
$\widetilde{\mathcal A}$ to $T_{b'}\widetilde{\mathcal B}' \times
T_{b'}\widetilde{\mathcal B}'$. In particular, $\{ v \in
T_{b'}\widetilde{\mathcal B}' : \sigma^{\mathcal B'}_{b'}(v,v)=0 \}
$ is contained in
\begin{eqnarray*}
&&\{ v \in T_{b'} \widetilde{\mathcal A} : \sigma_{b'}(v,v)=0\} \\
&=&\{ e'^* \otimes q + f'\otimes q^2 : \langle e'^*,f \rangle + \langle e^*,
f' \rangle =0 ,\,\, \langle e'^*,f' \rangle =0 ,\,\, e'^* \in E^*, f' \in E \}.
\end{eqnarray*}
Hence, $\{ v \in T_{b'}\widetilde{\mathcal B}' : \sigma^{\mathcal B'}_{b'}(v,v)=0 \}
$ is a linear section of an affine cone of a hyperquadric $\mathbb Q^3$
%$\{ e'^* \otimes q + f'\otimes q^2 : e'^* \in E^*, f' \in E \}$
because
$\mathcal B'$ is a linear section of $\mathcal A$.
For $b \in \widetilde{\mathcal B}_1$, $\{v
\in T_b\widetilde{\mathcal B} : \sigma^{\mathcal B}_b(v,v)=0 \}$ is
the union of three subspaces
$F^* \otimes q $, $e^* \otimes Q$ and ${F^*}^{\perp} \otimes q^2$,
while for $b' \in \widetilde{\mathcal B}_1' \cap (\widetilde{\mathcal A} \backslash   (E^* \otimes Q))$,
$\{ v \in T_{b'}\widetilde{\mathcal B}' : \sigma^{\mathcal B'}_{b'}(v,v)=0 \}$
is as above.
Thus the second fundamental form $\sigma^{\mathcal B}_{b}$ is not isomorphic to
$\sigma^{\mathcal B'}_{b'}$ and
hence $\mathcal B \subset \mathbb P(T_xZ)$
cannot be projectively equivalent to $\mathcal B' \subset \mathbb P(W')$, which is a contradiction.
Therefore, $\mathcal B_1'$ is contained in
$\mathbb P(E^* \otimes Q) \cap \mathcal A \simeq \mathbb P(E^*) \times \mathbb P(Q)$.

For $b'=e^*\otimes q \in \widetilde{\mathcal B'_1}$, $\{v
\in T_b'\widetilde{\mathcal B'} : \sigma^{\mathcal B}_b(v,v)=0 \}$ should be the union of three subspaces $R^* \otimes q $, $e^* \otimes Q$ and ${R^*}^{\perp} \otimes q^2$ for a subspaces $R^*\subset E^*$ of $\dim R^*=1$. Hence, $\mathcal B'$ has nonzero intersection $\mathcal B_1'=\mathbb P(R^*) \times
\mathbb P(Q)$ with $\mathbb P(E^*) \times
\mathbb P(Q)$.

Since $\mathcal B$ and $\mathcal B'$ are linear sections of $\mathcal A$
and second fundamental forms of $\mathcal B$ and $\mathcal B'$ are equivalent,
we see that $\{ v \in T_{b'}\widetilde{\mathcal B'} : \sigma_{b'}^{\mathcal
B'}(v,v)=0 \}$ is contained in $\{ v \in T_{b'
}\widetilde{\mathcal A} : \sigma_{b' }(v,v)=0 \}$ as a linear section;
\begin{eqnarray*}\{ e'^* \otimes q + f'\otimes q^2 : \langle e'^*,f \rangle + \langle e^*,
f' \rangle =0 ,\, e'^* \in R^*, f' \in {R^*}^{\perp} \}.\end{eqnarray*}
More precisely, for $b'=e^* \otimes q + f\otimes q^2 \in \tilde{\mathcal B'} \backslash \tilde{\mathcal B_1'} \subset\tilde{\mathcal A}\backslash (E^* \otimes Q)$, the space $\{ v \in T_{b'} \widetilde{\mathcal B'}: \sigma_{b'}(v,v)=0\}$ contains $\{ e'^* \otimes q + f'\otimes q^2 : \langle e'^*,f \rangle + \langle e^*,
f' \rangle =0 ,\, e'^* \in R^*, f' \in {R^*}^{\perp} \}$, meanwhile, for $b=e^* \otimes q + f\otimes q^2 \in \tilde{\mathcal B}\backslash \tilde{\mathcal B_1} \subset \tilde{\mathcal A} \backslash (E^* \otimes Q)$, the space $\{ v \in T_{b} \widetilde{\mathcal B}: \sigma_{b}(v,v)=0\}=\{ e'^* \otimes q + f'\otimes q^2 : \langle e'^*,f \rangle + \langle e^*,
f' \rangle =0 ,\, e'^* \in F^*, f' \in {F^*}^{\perp} \}$. Because second fundamental forms of $\mathcal B'$ and $\mathcal B$ are equivalent, the dimensions of base locus are same, the space $\{ v \in T_{b'} \widetilde{\mathcal B}: \sigma_{b'}(v,v)=0\}$ should be $\{ e'^* \otimes q + f'\otimes q^2 : \langle e'^*,f \rangle + \langle e^*,
f' \rangle =0 ,\, e'^* \in R^*, f' \in {R^*}^{\perp} \}$ which is tangent to the fiber of the projection $\mathcal A \rightarrow \mathbb P(Q)$. Hence, $\mathcal B'=\mathbb P \{ e^* \otimes q + f\otimes q^2 : q \in Q, e^* \in R^*, f \in {R^*}^{\perp} \}$.
Recall that $F^*$ and $R^*$ are linear subspaces in $E^*$ with the dimensions $\dim F^*=\dim R^*=1$.
%the dimensions $\dim F^*=\dim R^*=1$ of linear subspaces $F^*$ and $R^*$ of $E^*$.
Thus, $\mathcal B'=h \mathcal B$ for some $h \in \SL(E^*) \times \SL(Q)$, which is contained in the semisimple part of $P$.

For a smooth Schubert variety of type $(C_2, \alpha_2, \alpha_1)$,
we can prove this in the same way.
\end{proof}

\noindent {\it Proof of Theorem \ref{Main theorem}
in the case that $X$ is the $F_4$-homogeneous manifold associated to a short simple root}.
Since any smooth Schubert variety in the 15-dimensional $F_4$-homogeneous manifold $(F_4, \alpha_4)$ is
a linear space by Theorem 1.3 of Hong-Kwon \cite{HK},
it suffices to consider nonhomogeneous smooth Schubert varieties of type
$(B_3, \alpha_2, \alpha_3)$ and $(C_2, \alpha_2, \alpha_1)$
in the 20-dimensional $F_4$-homogeneous manifold $(F_4, \alpha_3)$.
Using Proposition \ref{automorphism groups of horosphercal}, Proposition \ref{nondegenerate: F_4},
Proposition \ref{tangency: F_4} and Proposition \ref{projective equivalence: F_4},
the same argument in the proof of Theorem \ref{Main theorem} given in Section 3 completes the proof.
\qed

\vskip 1em

\noindent
\textbf{Acknowledgements}. 
In October-November 2016, Workshop and International Conference on Spherical Varieties oganized by Michel Brion and Baohua Fu were held in Sanya.
This work had progressed while attending the workshop.
The authors would like to thank the organizers for their invitation and Tsinghua Sanya International Mathematics Forum for the support and hospitality.
They warmly thank Jaehyun Hong and Jun-Muk Hwang for the discussions on this topic and the useful comments. 
They also thank the referees for pointing out ambiguities and giving helpful comments. \\ %their time and comments. \\
This work was supported by National Researcher Program 2010-0020413 of NRF, GA17-19437S of GACR, Simons-Foundation grant 346300 and MNiSW 2015-2019 matching fund, BK21 PLUS SNU Mathematical Sciences Division and IBS-R003-Y1.

\end{document}